\documentclass[11pt]{amsart}

\usepackage[colorlinks=true,
pdfstartview=FitV, linkcolor=blue, citecolor=red,
urlcolor=cyan]{hyperref}
\usepackage{stmaryrd}
\usepackage{enumerate}
\usepackage[pdftex]{graphicx}
\usepackage{amsmath}
\usepackage{amssymb}
\usepackage{mathrsfs}
\usepackage{faktor}
\usepackage{mathdots}
\usepackage{amsthm}
\usepackage{relsize}
\usepackage{bbm}
\usepackage{color}
\usepackage[scale=0.8, top=2cm, bottom=2cm, left=2cm, right=2cm]{geometry}

\newtheorem{thm}{Theorem}[section]
\newtheorem{prop}[thm]{Proposition}
\newtheorem{lem}[thm]{Lemma}
\newtheorem{cor}[thm]{Corollary}
\newtheorem*{que*}{Question}

\theoremstyle{definition}
\newtheorem{definition}[thm]{Definition}
\newtheorem{example}[thm]{Example}

\theoremstyle{remark}
\newtheorem{notation}[thm]{Notation}
\newtheorem{remark}[thm]{Remark}

\numberwithin{equation}{section}

\newcommand{\absv}[1]{\left\lvert#1\right\rvert_v}

\newcommand{\Abb}{\mathbb A}
 
\newcommand{\Fbb}{\mathbb F}

\newcommand{\Rbb}{\mathbb R}
\newcommand{\Zbb}{\mathbb Z}

\newcommand{\Acal}{\mathcal A} 
\newcommand{\Bcal}{\mathcal B} 
 
\newcommand{\Ecal}{\mathcal E}

\newcommand{\Xscr}{\mathscr X}

\newcommand{\bott}{\mathrm{bot}}
\newcommand{\GL}{\mathrm{GL}}

\newcommand{\PGL}{\mathrm{PGL}}
\newcommand{\PSL}{\mathrm{PSL}}
\newcommand{\topp}{\mathrm{top}}

\newcommand{\ulim}{\mathrm{ulim}\ }

\thanks{This research was partially supported by the grant DMS-1406559 from the U.S. National Science Foundation. In addition, the author gratefully acknowledges support from the NSF grants DMS-1107452, 1107263 and 1107367 ``RNMS: GEometric structures And Representation varieties'' (the GEAR Network).}

\title[Positive configurations in a building and positive representations]{Positive configurations of flags in a building\\ and limits of positive representations}
\author{Giuseppe Martone}

\begin{document}
\maketitle
\noindent
\textbf{Abstract:} Parreau compactified the Hitchin component of a closed surface $S$ of negative Euler characteristic in such a way that a boundary point corresponds to the projectivized length spectrum of an action of $\pi_1(S)$ on an $\Rbb$-Euclidean building. In this paper, we use the positivity properties of Hitchin representations introduced by Fock and Goncharov to explicitly describe the geometry of a preferred collection of apartments in the limiting building.
\tableofcontents

\section{Introduction}\label{sec:intro}

Let $S$ be a connected, closed, oriented surface with negative Euler characteristic $\chi(S)$. The \emph{Teichm\"uller space} $\mathcal T(S)$ of $S$ is the space of isotopy classes of hyperbolic metrics on $S$. It is homeomorphic to $\Rbb^{-3\chi(S)}$. 

Thurston \cite{Thu, FLP} compactified $\mathcal T(S)$ in such a way that the resulting space $\overline{\mathcal T(S)}$ is homeomorphic to a closed ball of dimension $-3\chi(S)$. The boundary points of $\overline{\mathcal T(S)}$ can be described from different perspectives \cite{Bes, Bon2, MoSh, Pau1}. In particular, Morgan and Shalen used an algebro-geometric approach to realize these boundary points as length spectra of isometric actions of $\pi_1(S)$ on $\Rbb$-trees. An important point for their construction is that the Teichm\"uller space can be identified with a subspace of an affine variety. In fact, the holonomies of hyperbolic metrics let us realize $\mathcal T(S)$ as a connected component of the character variety
\[
\mathrm{Hom}(\pi_1(S),\PSL(2,\Rbb))/\!\!/\PSL(2,\Rbb).
\]
where $\PSL(2,\Rbb)$ acts by conjugation and we consider, as usual, the quotient in the sense of geometric invariant theory; see \cite{MFK} for details.

This description of $\mathcal T(S)$ is prone to generalizations. One can investigate subsets of different character varieties that share some of the properties of the Teichm\"uller space. For example, the natural action of $\mathrm{SL}(2,\Rbb)$ on the space of degree $d-1$ homogeneous real polynomials in two variables gives a homomorphism $\iota_d\colon \PSL(2,\Rbb)\to\PSL(d,\Rbb)$. Post-composing representations in the Teichm\"uller space with $\iota_d$ singles out a connected component
\[
(\iota_d)_*(\mathcal T(S))\subset \mathrm{Hit}_d(\Rbb)\subset \mathrm{Hom}(\pi_1(S),\PSL(d,\Rbb))/\!\!/\PSL(d,\Rbb).
\]
This \emph{Hitchin component} $\mathrm{Hit}_d(S)$ was identified and studied by Hitchin \cite{Hit} who proved that it is homeomorphic to $\Rbb^{-(d^2-1)\chi(S)}$. Using different methods, Fock and Goncharov \cite{FG1} and Labourie \cite{Lab} generalized many classical features of $\mathcal T(S)$ to the context of Hitchin representations. 

Much work has been done to describe generalized versions of Thurston's compactification for Hitchin components and related spaces \cite{Ale, BIPP, BP, CDLT, FG1, Le1, Le2, Par2, Par4}. 

The classical approach suggests the study of the \emph{(vector valued) length spectrum}
\[
L_d(\rho):=
(\log\lambda^\rho_1(\gamma), \log\lambda^\rho_2(\gamma),\dots, \log\lambda^\rho_d(\gamma))_{\gamma\in\pi(S)}.
\]
Here, $\lambda^\rho_i(\gamma)$ denotes the absolute value of the eigenvalues of $\rho(\gamma)$, which are non-zero and distinct \cite{FG1, Lab}. Usually, one also assumes that $\lambda^\rho_i(\gamma)> \lambda^\rho_{i+1}(\gamma)$. Parreau \cite{Par2} showed that the projectivized image of $L_d$ is relatively compact and that the boundary points of the closure can be realized as projectivized length spectra of isometric actions of $\pi_1(S)$ on an $\Rbb$-Euclidean building $\mathcal B_d$ of rank $d-1$. 

Euclidean buildings were introduced by Bruhat and Tits \cite{BT}. They are metric spaces equipped with an action of a reductive algebraic group over a field with discrete valuation. An $\Rbb$-Euclidean building is a generalization of an Euclidean building where the field is allowed to have a non-discrete valuation. 

For this introduction, it suffices to think of $\mathcal B_d$ as a generalization of an $\Rbb$-tree. It is a metric space obtained by gluing parametrized copies of the affine space
\[
\mathbb A^{d-1}:=\{(x_1,x_2,\dots,x_d)\in\Rbb^d\colon x_1+x_2+\dots+x_d=1\}
\]
called \emph{apartments}. Any two such parametrizations differ by an element of the \emph{affine Weyl group} $W_{\text{aff}}$, namely by the composition of a permutation of the coordinates and a translation by a vector in the underlying vector space which is naturally identified with 
\[
\mathbb V^{d-1}:=\{(x_1,x_2,\dots,x_d)\in\Rbb^d\colon x_1+x_2+\dots+x_d=0\}.
\]

The $\Rbb$-Euclidean building $\Bcal_d$ is associated to the general linear group over a specific field $\Fbb$ with valuation $v$. Each element of $\Fbb$ is an equivalence class of sequences of real numbers. The non-discrete valuation $v$ encodes information about the asymptotic behavior of such a sequence. Parreau \cite{Par1} described an explicit model for $\Bcal_d$ in which apartments correspond to line decompositions of a fixed $d$-dimensional $\Fbb$-vector space $V$. 

Our main contribution is to combine this explicit model and the positivity properties of Hitchin representations to describe the geometry of a preferred collection of apartments in the building $\mathcal B_d$. 

Let us be more explicit. Consider the universal cover of $\widetilde S$ and its boundary $\partial\widetilde S$. The choice of an auxiliary hyperbolic metric on $S$, identifies $\widetilde S$ to the hyperbolic plane and $\partial\widetilde S$ with the unit circle. For any $\rho\in\mathrm{Hit}_d(S)$, there exists a (unique up to $\PGL(d,\Rbb)$-action) $\rho$-equivariant map $\xi_\rho$ from $\partial\widetilde S$ into the space of complete flags in $\Rbb^d$ \cite{FG1, Lab}. 

This \emph{flag map} can be used to extend Thurston's parametrization of Teichm\"uller space via shearing coordinates to Hitchin components \cite{Thu2, Bon1, BoDr,BoDr2,FG1}. The idea is to fix a certain topological data on $S$ that singles out preferred tuples of distinct points in $\partial \widetilde S$. Using the flag map $\xi_\rho$, one then wishes to parametrize the space of tuples of flags in $\Rbb^d$ considered up to the action of $\PGL(d,\Rbb)$. Such a $\PGL(d,\Rbb)$ orbit is called a \emph{configurations of $t$ flags}. We will restrict our attention to tuples of flags that have the \emph{maximum span property} as in Definition \ref{def:maxspan}, which is a strong genericity condition.

It turns out that it is enough to consider configurations of three and four flags, which can be parametrized by two families of real numbers. Any orbit of four flags $(E,F,G,H)$ that have the maximum span property has associated \emph{triple ratios} $X_{a,b,c}(E,F,G)$ and $X_{a,b,c}(E,G,H)$ and \emph{double ratios} $Z_i(E,F,G,H)$, where $a,b,c\geq 1$ are integers such that $a+b+c=d$, and $i=1,2, \dots, d-1$.

Fock and Goncharov show that for any $\rho\in\mathrm{Hit}_d(S)$, the images of tuples of distinct points via $\xi_\rho$ are \emph{positive} in the following sense. For any three distinct points $x_1,x_2,x_3\in\partial \widetilde S$, the triple ratios of $(\xi_\rho(x_1),\xi_\rho(x_2),\xi_\rho(x_3))$ are positive. Moreover, for any four distinct points $x_1,x_2,x_3,x_4$ in this cyclic order along $\partial \widetilde S$, the double ratios $Z_i(\xi_\rho(x_1),\xi_\rho(x_2),\xi_\rho(x_3),\xi_\rho(x_4))$ are positive. 

In this paper we use this positivity property of Hitchin representations to describe intersections of apartments in the $\Rbb$-Euclidean building $\mathcal B_d$ arising as limits of positive tuple of flags in $\Rbb^d$. 

Consider a sequence of positive tuples of flags $(F_{1,n},F_{2,n},\dots, F_{t,n})$ in $\Rbb^d$. It follows from Lemma \ref{lem:ulimflag} and from the definition of the field $\Fbb$ that there exists a unique limiting tuple of flags $(F_1, F_2,\dots, F_t)$ in $\Fbb^d$ that we call the \emph{ultralimit of $(F_{1,n},F_{2,n},\dots, F_{t,n})$}. This tuple of flags $(F_1, F_2,\dots, F_t)$ in $\Fbb^d$ is \emph{positive} if it has the maximum span property and if the sequences of Fock-Goncharov parameters of the tuple $(F_{1,n},F_{2,n},\dots, F_{t,n})$ define non-zero elements in the field $\Fbb$. The genericity condition guarantees that any two such flags $F_i$ and $F_j$ in $\Fbb^d$ determine an apartment in the $\Rbb$-Euclidean building $\mathcal B_d$.

\begin{thm}\label{intro:triple} Let $(E_n,F_n,G_n)$ be a sequence of positive triples of flags in $\Rbb^d$. Assume that its ultralimit $(E,F,G)$ is positive. Denote by $\mathcal A_{EF}$, $\mathcal A_{FG}$, and $\mathcal A_{EG}$ the apartments in $\mathcal B_d$ corresponding to the pairs $(E,F)$, $(F,G)$ and $(E,G)$, respectively. Then, there exists 
\begin{itemize}
	\item[-] a preferred parametrization $f_{EG}\colon\Abb^{d-1}\to\mathcal A_{EG}$,
	\item[-] two closed cones $\mathfrak C_1$ and $\mathfrak C_2$ in $\Abb^{d-1}$, defined by the inequalities \ref{eqn:mainthm1} and \ref{eqn:mainthm2},
\end{itemize}
such that 
\begin{gather*}
\mathfrak C_1=f_{EG}^{-1}(\mathcal A_{EG}\cap\mathcal A_{EF})\text{ and }\ \mathfrak C_2=f_{EG}^{-1}(\mathcal A_{EG}\cap\mathcal A_{FG}).
\end{gather*}
The cones $\mathfrak C_1$ and $\mathfrak C_2$ are described explicitly in terms of the valuations of the sequences of triple ratios $(X_{a,b,c}(E_n,F_n,G_n))$.
\end{thm}

In the statement of Theorem \ref{intro:triple}, one can permute the three flags $E$, $F$ and $G$ to obtain similar descriptions of the intersections of the apartments $\mathcal A_{EF}$, $\mathcal A_{FG}$, and $\mathcal A_{EG}$ in terms of preferred parametrizations of the apartment $\Acal_{FG}$ or of the apartment $\Acal_{EG}$.

In particular, applying Theorem \ref{intro:triple} to the sequences of positive triple of flags $(E_n,F_n,G_n)$ and $(E_n,G_n,H_n)$, we obtain two parametrization $f_{EG}$ and $f'_{EG}$ for the apartment $\Acal_{EG}$. As observed above, it is a consequence of the definition of an $\Rbb$-Euclidean building that these two parametrizations differ by an element $w_{(E,F,G,H)}$ of the affine Weyl group $W_{\text{aff}}$.

\begin{thm}\label{intro:quadruple} Consider a sequence $(E_n,F_n,G_n,H_n)$ of positive quadruples of flags in $\Rbb^d$. Assume that its ultralimit $(E,F,G,H)$ is positive. Denote by $f_{EG}$ and $f'_{EG}$ the preferred parametrizations of the apartment $\mathcal A_{EG}$ obtained by applying Theorem \ref{intro:triple} to the sequences of positive triples of flags $(E_n,F_n,G_n)$ and $(E_n,G_n, H_n)$. Then, the element
\[
w_{(E,F,G,H)}:=f_{EG}^{-1}\circ f'_{EG}\in W_{\text{aff}}
\]
of the affine Weyl group $W_{\text{aff}}$ is the translation of $\Abb^{d-1}$ by the unique vector $(x_1,x_2,\dots,x_d)$ in $\mathbb V^{d-1}$ such that the difference $x_{i+1}-x_i$ is the valuation of the element in $\Fbb$ defined by the sequence of double ratios $Z_{d-i}(E_n,F_n,G_n,H_n)$. 
\end{thm}

Theorem \ref{intro:triple} and Theorem \ref{intro:quadruple} are well-known for $d=2$ and in the case of $d=3$ they follow from \cite[Thm. 1]{Par3} and \cite[Prop. 4.5]{Par4}, respectively. However, we emphasize that our results are obtained via a different approach that only relies on the positivity properties of the sequences of flags. 

An immediate consequence of our explicit formulas is that the triple intersection $\mathcal A_{EF}\cap \mathcal A_{FG}\cap\mathcal A_{EG}$ is at most one point and it is exactly one point if the valuations of all the sequences of triple ratios are zero.

Our next result concerns the geometry of apartments in the $\Rbb$-Euclidean building $\Bcal_d$ associated to the ultralimit of a sequence of positive tuples of flags for $t\geq 4$. Consider a sequence $(F_{1,n}, F_{2,n},\dots, F_{t,n})$ of positive tuples of flags in $\Rbb^d$, and assume that the ultralimit $(F_1,F_2,\dots, F_t)$ is positive. Then, there exists an apartment $\Acal_{ij}$ in $\Bcal_d$ associated to each pair of flags $(F_i,F_j)$. We say that such an apartment $\Acal_{i_2,j_2}$ \emph{combinatorially separates} the apartments $\Acal_{i_1j_1}$ and $\Acal_{i_3j_3}$ if, up to a cyclic permutation of the indices of the flags $(F_1,F_2,\dots, F_t)$, we have that
\[
1\leq i_1 \leq i_2 \leq i_3< j_3\leq j_2\leq j_1\leq t.
\]
\begin{thm}[Monotonicity]\label{intro:mono} Let $(F_{1,n}, F_{2,n},\dots, F_{t,n})$ be a sequence of positive tuples of flags in $\Rbb^d$ with positive ultralimit the tuple of flags $(F_1,F_2 ,\dots, F_t)$ in $\Fbb^d$. Let $\Acal_1$, $\Acal_2$ and $\Acal_3$ be apartments in the $\Rbb$-Euclidean building $\Bcal_d$ corresponding to a pairs of flags $(F_{i_1},F_{j_1})$, $(F_{i_2},F_{j_2})$ and $(F_{i_3},F_{j_3})$, respectively. If the apartment $\Acal_2$ combinatorially separates the apartments $\Acal_1$ and $\Acal_3$, then
\[
\Acal_1\cap\Acal_3=\Acal_1\cap\Acal_2\cap\Acal_3.
\]
\end{thm}

In other words, Theorem \ref{intro:mono} relates combinatorial separation, which is a property depending exclusively on the cyclic order of the tuple of flags $(F_1,F_2,\dots, F_t)$, to intersection properties of the corresponding apartments in the $\Rbb$-Euclidean building $\Bcal_d$.

\medskip\noindent
\textbf{Acknowledgments:} It is a pleasure to thank my thesis advisor, Francis Bonahon, for encouraging me to think about this problem, for the numerous insightful conversations, and for his support. I thank Daniele Alessandrini and Beatrice Pozzetti for useful discussions and feedback. I am very grateful to the referee for providing several useful comments on an earlier version of this manuscript.

\section{Flags, snakes and positivity}\label{sec:flagsPos}

\subsection{Configurations of $t$ flags and their parametrization}\label{sec:tripledouble}

A \emph{(complete) flag} $F$ in $\Rbb^d$ is a nested sequence 
\[
0=F^{(0)}\subset F^{(1)}\subset\dots\subset F^{(d-1)}\subset  F^{(d)}=\Rbb^{d}
\]
of vector subspaces of $\Rbb^d$  such that $\dim F^{(i)}=i$ for all $i$. The quotient $\PGL(d,\Rbb)$ of the general linear group $\mathrm{GL}(d,\Rbb)$ by the subgroup of non-zero scalar matrices acts naturally on the space of flags.

We focus on tuples of flags enjoying the following genericity property.

\begin{definition}\label{def:maxspan} The tuple of flags $(F_1,F_2,\dots, F_t)$ has the \emph{maximum span property} if for any integers $0\leq a_1,a_2,\dots,a_t\leq d$ the following equality holds
\begin{equation}\label{eqn:maxspan}
\dim\left(F_1^{(a_1)}+F_2^{(a_2)}+\dots +F_t^{(a_t)}\right)=\min\{a_1+a_2+\dots+a_t,d\}.
\end{equation}
\end{definition}

Observe that the diagonal action of $\PGL(d,\Rbb)$ on the space of tuples of flags preserves the maximum span property. 

\begin{definition}\label{def:confflags} A  \emph{configuration of $t$ flags} is a tuple of flags with the maximum span property considered up to the diagonal action of $\PGL(d,\Rbb)$. Denote by $\mathscr X_t$ the space of configurations of $t$ flags. 
\end{definition}

It follows from elementary linear algebra that $\Xscr_2$ is a single point. Henceforth, we assume $t>2$. In this case there are several $\PGL(d,\Rbb)$ orbits of maximum span tuples of flags.

Fock and Goncharov \cite{FG1} parametrized preferred subspaces of $\mathscr X_t$. The Fock-Goncharov coordinates are expressed in terms of the wedge products of vectors in $\Rbb^{d}$. It is convenient to fix once and for all an identification $\bigwedge^d\Rbb^d\cong\Rbb$ and to observe the following.

\begin{remark}\label{rmk:maxspan} Let $(F_1,F_2,\dots, F_t)$ be a tuple of flags with the maximum span property and let $a_1,a_2,\dots, a_t\geq 0$ be integers such that $a_1+a_2+\dots +a_t=d$. Choose non-zero elements 
\begin{gather*}
f_j^{(a_j)}\in\bigwedge^{a_j}F_j^{(a_j)}\subset \bigwedge^{a_j}\Rbb^d.
\end{gather*}
The maximum span property guarantees that $f_1^{(a_1)}\wedge f_2^{(a_2)}\wedge \dots\wedge f_t^{(a_t)}$ is different from zero.
\end{remark}

\subsubsection{Triple ratios}\label{sssec:tripleratios}

Consider the discrete triangle
\[
\Theta_d:=\{(a,b,c)\in\Zbb^3\colon a+b+c=d,\ a,b,c\geq 0\}
\] 
depicted in Figure \ref{fig:Theta_n} and its interior
\[
\Theta_d^\circ:=\{(a,b,c)\in\Zbb^3\colon a+b+c=d,\ a,b,c> 0\}.
\]

\begin{figure}
\includegraphics[scale=.4]{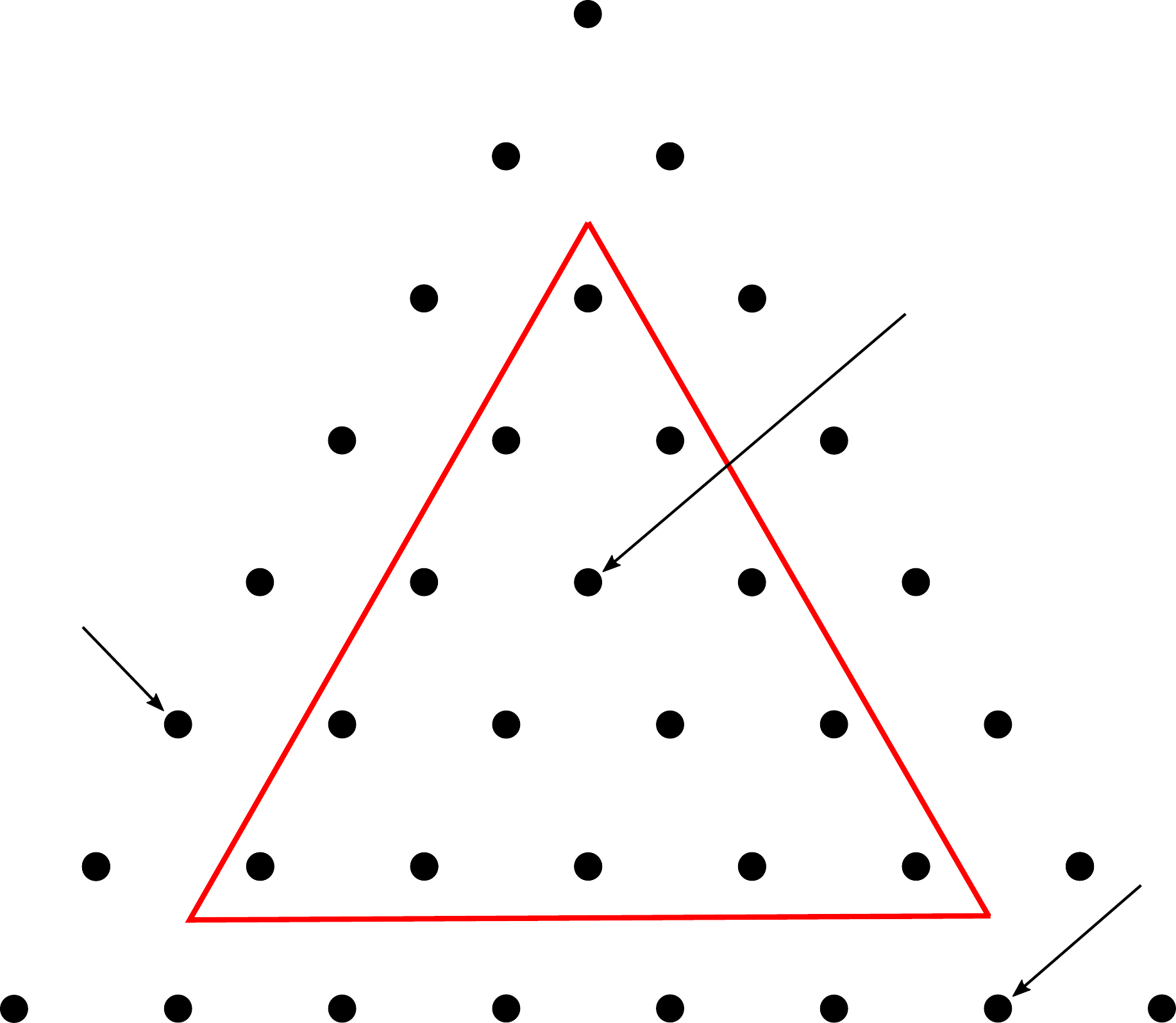}
\put (-530, 74){\makebox[0.7\textwidth][r]{\small{$(5,2,0)$}}}
\put (-367,131){\makebox[0.7\textwidth][r]{\small{$(2,3,2)$}}}
\put (-327, 28){\makebox[0.7\textwidth][r]{\small{$(1,0,6)$}}}
\caption{The discrete triangle $\Theta_d$ for $d=7$. Highlighted, its interior $\Theta_d^\circ$.}
\label{fig:Theta_n}
\end{figure}

\begin{definition}\label{def:tripleratio} Let $(E,F,G)$ be a triple of flags with the maximum span property. For $(a,b,c)\in\Theta_d^\circ$, define the \emph{$(a,b,c)$-triple ratio of $(E,F,G)$} as 
\begin{align*}
X_{a,b,c}(E,F,G):=&\frac{e^{(a-1)}\wedge f^{(b)}\wedge g^{(c+1)}}{e^{(a+1)}\wedge f^{(b)}\wedge g^{(c-1)}}\cdot\frac{e^{(a)}\wedge f^{(b+1)}\wedge g^{(c-1)}}{e^{(a)}\wedge f^{(b-1)}\wedge g^{(c+1)}} \cdot \frac{e^{(a+1)}\wedge f^{(b-1)}\wedge g^{(c)}}{e^{(a-1)}\wedge f^{(b+1)}\wedge g^{(c)}}
\end{align*}
where we chose non-zero vectors $e^{(\cdot)}$, $f^{(\cdot)}$, and $g^{(\cdot)}$ in the exterior powers $\bigwedge^\cdot E^{(\cdot)}$, $\bigwedge^\cdot F^{(\cdot)}$, and $\bigwedge^\cdot G^{(\cdot)}$, respectively.
\end{definition}

The triple ratios do not depend on any of the choices made in the definition and Remark \ref{rmk:maxspan} guarantees that they are non-zero real numbers. The triple ratios are constant on $\PGL(d,\Rbb)$ orbits.

\begin{thm}\label{thm:tripleratios} The map assigning the triple ratios to a configuration of three flags is a bijection between $\Xscr_3$ and $(\Rbb-\{0\})^\frac{(d-1)(d-2)}{2}$.
\end{thm}
\begin{proof} Cf. \cite[\S 9]{FG1}.
\end{proof}

\begin{remark}\label{rmk:symtriple} If we permute the flags $E$, $F$ and $G$, the triple ratios vary according to the formulas
\[
X_{a,b,c}(E,F,G)=X_{b,c,a}(F,G,E)=X_{b,a,c}(F,E,G)^{-1}.
\]
\end{remark}

\subsubsection{Double ratios}\label{sssec:doubleratios}

In the case of four flags, one needs to consider a generalized version of the classical cross ratio of four points on a projective line.

\begin{definition}\label{def:doubleratios} Let $(E,F,G,H)$ be a quadruple of flags with the maximum span property. For $0< i< d$, the \emph{$i$-double ratio} is
\[
Z_i(E,F,G,H)=-\frac{e^{(i)}\wedge g^{(d-i-1)}\wedge f^{(1)}}{e^{(i)}\wedge g^{(d-i-1)}\wedge h^{(1)}}\cdot\frac{e^{(i-1)}\wedge g^{(d-i)}\wedge h^{(1)}}{e^{(i-1)}\wedge g^{(d-i)}\wedge f^{(1)}}.
\]
where we chose non-zero vectors $e^{(\cdot)}$, $f^{(1)}$, $g^{(\cdot)}$, and $h^{(1)}$ in $\bigwedge^\cdot E^{(\cdot)}$, $F^{(1)}$, $\bigwedge^\cdot G^{(\cdot)}$, and $H^{(1)}$ respectively.
\end{definition}

Note that the double ratios do not depend on the choices involved in the definition and Remark \ref{rmk:maxspan} implies that they are non-zero real numbers. The double ratios are constant on the $\PGL(d,\Rbb)$ orbit of $(E,F,G,H)$. 

\begin{remark}\label{rmk:symdouble} The r\^oles of the flags $(E,G)$ and $(F,H)$ in the definition of the double ratios are not equal. Explicit computations show that if we consider permutations of $E$, $F$, $G$ and $H$ that respect this lack of symmetry (called \emph{dihedral permutations}) the corresponding double ratios are related to the original ones by the formulas
\[
Z_i(E,F, G,H)=Z_{d-i}(G,H,E,F)=Z_i(E,H,G,F)^{-1}.
\]
\end{remark}

\begin{thm}\label{thm:doubleatios} The configuration of four flags $(E,F,G,H)\in\Xscr_4$ is determined by the data of
\begin{itemize}
	\item[-] the triple ratios $X_{a,b,c}(E,F,G)$ for $(a,b,c)\in\Theta^\circ_d$;
	\item[-] the triple ratios $X_{a,b,c}(E,H,F)$ for $(a,b,c)\in\Theta^\circ_d$;
	\item[-] the double ratios $Z_i(E,F,G,H)$ for $0< i< d$.
\end{itemize}
\end{thm}
\begin{proof} Cf. \cite[\S5 and \S9]{FG1}.
\end{proof}

\subsection{Snakes and their moves}\label{ssec:snakes}

In this section we describe how the triple and double ratios encode information about the linear algebra of a quadruple of maximum span flags. We follow the exposition in \cite[App. A]{GMN}.

\begin{notation} Let us ease notation for the rest of this section by fixing a maximum span quadruple of flags $(E,F,G,H)$ and by setting $X_{a,b,c}:=X_{a,b,c}(E,F,G)$ and $Z_i:=Z_i(E,F,G,H)$. 
\end{notation}

\subsubsection{Snakes}\label{sssec:snakes}

The \emph{dual triangle of $\Theta_d$} is the discrete triangle $\Theta^\perp_d= \Theta_{d-1}$ where a point $(\alpha,\beta,\gamma)\in\Theta^\perp_d$ corresponds to the triangle in $\Theta_d$ with vertices
\[
(\alpha+1,\beta,\gamma),\ (\alpha,\beta+1,\gamma),\ (\alpha,\beta,\gamma+1).
\] 
See Figure \ref{fig:dualtriangle}.

\begin{figure}
\includegraphics[scale=0.4]{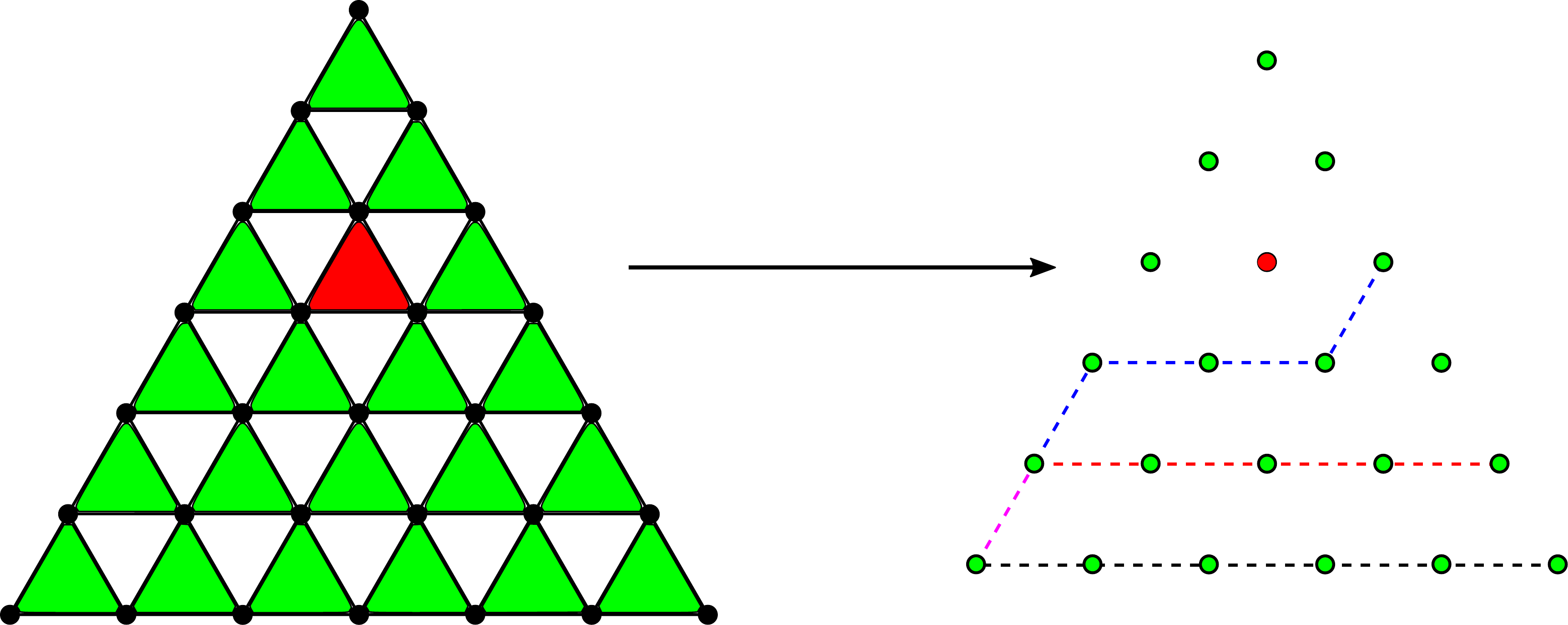}
\put (-408, 17){\makebox[0.7\textwidth][r]{\small$\sigma^\bott$}}
\put (-407, 42){\makebox[0.7\textwidth][r]{\small$\sigma$}}
\put (-530, 93){\makebox[0.7\textwidth][r]{\small$\perp$}}
\caption{The triangle $\Theta_d$ and its dual $\Theta^\perp_d$ for $d=6$. The red triangle on the left corresponds to the red dot on the right. The dashed lines trace examples of snakes in $\Theta^\perp_d$.}
\label{fig:dualtriangle}
\end{figure}

\begin{definition} A \emph{snake $\sigma$ in $\Theta^\perp_d$} is a sequence of $d$ points $\sigma(k)=(\alpha_k,\beta_k,\gamma_k)\in\Theta^\perp_d$ such that $(\alpha_1,\beta_1,\gamma_1)=(d-1,0,0)$ and
\[
(\alpha_{k+1},\beta_{k+1},\gamma_{k+1})=(\alpha_k-1,\beta_k+1,\gamma_k)\text{ or }(\alpha_k-1,\beta_k,\gamma_k+1).
\]
\end{definition}
\begin{example}\label{ex:topbotsnakes} The \emph{top snake} of $\Theta^\perp_d$ is $\sigma^\topp(k)=(d-k,k-1,0)$; the \emph{bottom snake} of $\Theta^\perp_d$ is $\sigma^\bott(k)=(d-k,0,k-1)$.
\end{example}

For a subspace $W\subset\Rbb^d$, the \emph{dual of $W$} is the vector space 
\[
W^\perp:=\{u\in (\Rbb^d)^*\colon u_{\vert_W}=0\}\subset (\Rbb^d)^*.
\] 
Note that $\dim W+\dim W^\perp=d$. For a flag $F\in\mathrm{Flag}(\Rbb^d)$, the \emph{dual flag} $F^\perp\in\mathrm{Flag}((\Rbb^d)^*)$ is defined by
\[
(F^\perp)^{(i)}:=(F^{(d-i)})^\perp.
\]

The data of a snake and a maximum span triple of flags $(E,F,G)$ determines a projective basis for $(\mathbb R^d)^*$ as follows. Given a snake $\sigma(k)=(\alpha_k,\beta_k,\gamma_k)$ the one dimensional subspaces $L_k:=(E^{(\alpha_k)}+F^{(\beta_k)}+G^{(\gamma_k)})^\perp$ form a line decomposition of $(\Rbb^d)^*=\oplus_{i=1}^dL_i$. Choose a non-zero vector $u_1\in L_1=(E^{(d-1)})^\perp$. We wish to inductively define a non-zero vector $u_i\in L_i$ for all $i>1$. Assume we have defined $u_k\in L_k$. Given $\sigma(k)=(\alpha_k,\beta_k,\gamma_k)$, there are two options for the value of $\sigma(k+1)$ or, equivalently, for the line $L_{k+1}$:
\begin{gather*}
L_{k+1}'=\big(E^{(\alpha_k-1)}+F^{(\beta_k)}+G^{(\gamma_k+1)}\big)^\perp\ \text{ or }\ L_{k+1}''=\big(E^{(\alpha_k-1)}+F^{(\beta_k+1)}+G^{(\gamma_k)}\big)^\perp.
\end{gather*}
\begin{lem}\label{lem:snakebase} For $u_k\in L_k$, there exist unique $u'_{k+1}\in L_{k+1}'$ and $u_{k+1}''\in L_{k+1}''$ so that $u_k+u_{k+1}'+u_{k+1}''=0$.
\end{lem}
\begin{proof} Cf. \cite[\S A.4]{GMN}.
\end{proof}
The desired basis is obtained by setting
\begin{equation}\label{eqn:snakebasis}
u_{k+1}=
\begin{cases}
u_{k+1}'&\text{ if }\sigma(k+1)=(\alpha_k-1,\beta_k,\gamma_k+1)\\
-u_{k+1}''&\text{ if }\sigma(k+1)=(\alpha_k-1,\beta_k+1,\gamma_k)
\end{cases}
\end{equation}
The choice of sign in Equation \ref{eqn:snakebasis} will be justified in \S \ref{ssec:positivity}. If we replace $u_1$ with $u'_1=\lambda u_1$ for some $\lambda\neq 0$, the corresponding basis $(u'_i)$ will be so that $u'_i=\lambda u_i$. Therefore, a snake determines via this construction a unique projective basis that we refer to as its \emph{snake basis}.

\subsubsection{Snake bases changes}\label{sssec:snakechange}

A snake can ``move''  in two basic ways.

\begin{definition} Let $\sigma$ and $\sigma'$ be snakes in $\Theta^\perp_d$.
\begin{itemize}
	\item[-] $\sigma'$ is obtained from $\sigma$ by a \emph{tail move} if $\sigma$ and $\sigma'$ only differ in the $d$-th position so that if $\sigma(d)=(0,\beta_d,\gamma_d)$, then $\sigma'(d)=(0,\beta_d+1,\gamma_d-1)$.
	\item[-] $\sigma'$ is obtained from $\sigma$ by a \emph{diamond move at $k+1$}, with $k<d-1$, if $\sigma$ and $\sigma'$ only differ in the $(k+1)$-st position so that if $\sigma(k+1)=(\alpha_{k+1},\beta_{k+1},\gamma_{k+1})$, then $\sigma'(k+1)=(\alpha_{k+1},\beta_{k+1}+1,\gamma_{k+1}-1)$.
\end{itemize}
\end{definition}

More explicitly, if $\sigma'$ is obtained from $\sigma$ by a diamond move at $k+1$ we have
\begin{gather*}
\sigma(i)=\sigma'(i)=(\alpha_i,\beta_i,\gamma_i)\text{ for }i\neq k+1,\\
\sigma(k+1)=(\alpha_k-1,\beta_k,\gamma_k+1),\\
\sigma'(k+1)=(\alpha_k-1,\beta_k+1,\gamma_k).
\end{gather*}

\begin{example}\label{ex:snakemoves} The snake $\sigma$ in Figure \ref{fig:dualtriangle} is obtained from $\sigma^\bott$ by a tail move and diamond moves at $k+1$ for $k=4,3,2,1$.
\end{example}

The next Proposition relates the triple ratio to snake bases and moves.

\begin{prop}[Snake moves]\label{prop:snakemoves} Let $\sigma$ and $\sigma'$ be snakes in $\Theta_d^\perp$. Denote by $(u'_i)$ and $(u_i)$ the respective snake bases. Suppose $u_1=u_1'$.
\begin{itemize}
	\item[-] If $\sigma'$ is obtained from $\sigma$ by a tail move, then
	\begin{equation}\label{eqn:tail}
	u'_i=
	\begin{cases}
	u_i & \text{ for } i< d\\
	u_{i-1}+u_i & \text{ for } i=d
	\end{cases}
	\end{equation}
	\item[-] If $\sigma'$ be a snake obtained from $\sigma$ by a diamond move at $k+1$, then
	\begin{equation}\label{eqn:diamond}
	u'_i=
	\begin{cases}
	u_i & \text{ for } i< k+1\\
	u_{i-1}+u_i & \text{ for } i= k+1 \\
	X_{(\alpha_k-1)(\beta_k+1)(\gamma_k+1)}u_i & \text{ for } i> k+1
	\end{cases}
	\end{equation}
\end{itemize}
\end{prop}
\begin{proof} Cf. \cite[\S 9]{FG1}. See also \cite[\S A.4]{GMN}.
\end{proof}

Fix any pair of snakes $\sigma$ and $\sigma'$ and respective snake bases so that $u_1=u_1'$. We denote by $\mathcal M_\sigma^{\sigma'}(E,F,G)\in\GL(d,\Rbb)$ the upper triangular basis change matrix between the snake bases of $\sigma$ and $\sigma'$. It is a product of (some of) the basis change matrices described in Proposition \ref{prop:snakemoves}.

\subsubsection{Shearing}\label{sssec:shearing}
Double ratios can also be understood in terms of snake bases. In fact, the maximum span quadruple of flags $(E,F,G,H)$ determines two projective basis $(u_i)$ and $(U_i)$ corresponding to the line decomposition $L_i=(E^{(d-i)}\oplus G^{(i-1)})^\perp$: the projective basis $(u_i)$ defined via the maximum span triple $(E,F,G)$ and the projective basis $(U_i)$ defined via $(E,G,H)$. The following well-known proposition relates these two bases and the double ratios of the quadruple $(E,F,G,H)$. We include a proof for completeness.

\begin{prop}\label{prop:shearing} Let $(u_i)$ and $(U_i)$ be as above. Assume $U_d=u_d$. Then
\begin{equation}\label{eqn:shearing}
U_i=Z_1\cdots Z_{d-i}u_i,\text{ for }\ 0<i<d
\end{equation}
\end{prop}
\begin{proof} Let $(e_i)$ denote the standard basis of $\Rbb^d$. Up to $\PGL(d,\Rbb)$ action, we can renormalize the flags $(E,F,G)$ so that
\begin{align*}
E^{(i)}&=\mathrm{Span}(e_1, e_2,\dots, e_i);\\
G^{(i)}&=\mathrm{Span}(e_d, e_{d-1},\dots, e_{d-i+1});\\
F^{(1)}&=\mathrm{Span}(e_1+e_2+\dots+e_d).
\end{align*}
Pick a non-zero vector $h_1e_1+h_2e_2+\dots+h_de_d\in H^{(1)}$. Note that the maximum span property implies that $h_i\neq 0$ for all $i=1,2,\dots, d$. By Definition \ref{def:doubleratios}, we compute the double ratios to be
\[
Z_i=-\frac{h_i}{h_{i+1}}, \ i=1,2,\dots,d-1.
\]
Denote by $e_i^t$ the transpose of the vector $e_i$ seen as an element in $(\Rbb^d)^*$. Note that $(E^{(d-1)})^\perp=\mathrm{Span}(e_{d}^t)$ and for $i>1$
\begin{align*}
(E^{(d-i)}\oplus G^{(i-1)})^\perp&=\mathrm{Span}(e_{d-i+1}^t);\\
(E^{(d-i)}\oplus F^{(1)} \oplus G^{(i-2)})^\perp&=\mathrm{Span}(e_{d-i+1}^t-e_{d-i+2}^t).
\end{align*} 
Likewise, note that 
\begin{align*}
(E^{(d-i)}\oplus G^{(i-2)}\oplus H^{(1)})^\perp&=\mathrm{Span}(v)
\end{align*}
with $v\in \mathrm{Span}(e^t_{d-i+1},e^t_{d-i+2})$ and $H^{(1)}\subset \ker(v)$. A computation shows that $v$ is a multiple of the vector $e^t_{d-i+1}+Z_{d-i+1}e^t_{d-i+2}$.

The vectors $u_i=\lambda_ie_{d-i+1}^t$ are defined recursively by solving Equation \ref{eqn:snakebasis}. Namely,
\[
\lambda_{i-1}e_{d-(i-1)+1}^t+\lambda_ie_{d-i+1}^t+\mu_i(e_{d-i+1}^t-e_{d-i+2}^t)=0
\]
for some $\lambda_{i-1},\lambda_i,\mu_i\in\Rbb-\{0\}$. If we set $\lambda_1=1$ and proceed by induction, we obtain $u_i=(-1)^{i-1}e_{d-i+1}^t$. 

The vectors $U_i=-\Lambda_ie_{d-i+1}^t$, on the other hand, are defined recursively by Equation \ref{eqn:snakebasis} as the unique solutions to
\[
-\Lambda_{i-1}e_{d-(i-1)+1}^t+\Lambda_ie_{d-i+1}^t+M_i(e^t_{d-i+1}+Z_{d-i+1}e^t_{d-i+2})=0
\]
for $\Lambda_{i-1},\Lambda_i,M_i\in\Rbb-\{0\}$ were we assume $\Lambda_d=(-1)^{d-2}$ so that $u_d=U_d$. By iteration, we obtain the following equality for all $i\leq d$
\[
\Lambda_{i-1}e_{d-(i-1)+1}^t+(-1)^{i-1}Z_1\dots Z_{d-i}e_{d-i+1}^t=M_i(e^t_{d-i+1}+Z_{d-i+1}e^t_{d-i+2}).
\]
Because $\Lambda_{i-1}=M_iZ_{d-i+1}$, it follows that $U_{i-1}=-(-1)^{i-1}Z_1\dots Z_{d-i}Z_{d-i+1}e_{d-i+2}^t$, as needed.
\end{proof}

We denote by $\mathcal S(E,F,G,H)\in \GL(d,\Rbb)$ the basis change matrix described by Proposition \ref{prop:shearing}.

\subsection{Positive configurations of flags and total positivity}\label{ssec:positivity}

Total positivity was introduced by Gantmacher and Krein \cite{GKr} and Schoenberg \cite{Sch} for matrices in $\GL(d,\Rbb)$. 

\begin{definition}\label{def:realtotpos} A matrix $M$ in $\GL(d,\Rbb)$ is \emph{totally nonnegative} if all of its minors are greater or equal to zero. An element $M$ in $\GL(d,\Rbb)$ is \emph{totally positive} if all of its minors belong to $\Rbb_{> 0}$.
\end{definition}

\begin{example} The matrix $\begin{pmatrix}1 & 1 & 1\\0 & 1 & 1+X\\0 & 0 & X \end{pmatrix}$ is totally nonnegative whenever $X>0$.
\end{example}

Note that the product of totally nonnegative (resp. positive) matrices is totally nonnegative (resp. positive). 

Definition \ref{def:realtotpos} has been greatly extended by Lusztig \cite{Lus} and it plays a prominent r\^{o}le in \cite{FG1}. In fact, total positivity arises in the context of configurations of flags as follows. 

Denote by $\mathscr P_t$ the regular convex polygon with $t$ vertices $v_1,v_2,\dots, v_t$ appearing in this clockwise order along the boundary of $\mathscr P_t$. A tuple of flags $(F_1,F_2,\dots, F_t)$ determines a natural labeling of the vertices of $\mathscr P_t$: the vertex $v_i$ corresponds to the flag $F_i$. An \emph{oriented triangulation} $\mathscr T$ of $\mathscr P_t$ is a collection of oriented edges $(v_i,v_k)$ that subdivide $\mathscr P_t$ into triangles. We label any such triangle by its vertices $(v_i,v_j, v_k)$ where, by convention, we assume $i< j< k$.  An \emph{internal edge} of $\mathscr T$ is an edge of the triangulation that does not belong to the boundary of $\mathscr P_t$. Any internal edge $(v_i,v_k)$ is a diagonal for a quadrilateral $(v_i, v_j, v_k, v_l)$ where the vertices appear in this cyclic order around $\mathscr P_t$. Any triangulation $\mathscr T$ has $t-3$ internal edges and it subdivides $\mathscr P_t$ into $t-2$ triangles.


One can use the Fock-Goncharov coordinates introduced in \S \ref{sssec:tripleratios} and \S \ref{sssec:doubleratios} and the oriented triangulation $\mathscr T$ to define coordinates for points in $\mathscr X_t$. Moreover, one can define $\mathscr X_t^+(\mathscr T)$ as the subset of $\mathscr X_t$ of configurations of $t$ flags whose coordinates with respect to $\mathscr T$ are positive. 

\begin{thm}\label{thm:positivity} Let $\mathscr T$ and $\mathscr T'$ be any two oriented triangulations of the regular convex polygon with $t$ vertices $\mathscr P_t$. Then, 
\[
\mathscr X_t^+(\mathscr T)=\mathscr X_t^+(\mathscr T')\cong \mathbb R_{>0}^N,
\]
where $N=\frac{(d-2)(d-1)}{2}(t-2)+(d-1)(t-3)$.
\end{thm}
\begin{proof} Cf. \cite[Thm. 1.5]{FG1}.
\end{proof}

Theorem \ref{thm:positivity} justifies the following definition.

\begin{definition}\label{def:posconfflags} A \emph{positive configuration of $t$ flags} is a configuration of $t$ flags that belongs to $\mathscr X_t^+(\mathscr T)$ for any (all) oriented triangulation(s) $\mathscr T$ of the polygon $\mathscr P_t$. Denote by $\mathscr X^+_t$ the space of positive configurations of $t$ flags. A tuple of flags $(F_1,F_2,\dots, F_t)$ in $\Rbb^d$ is \emph{positive} if its $\PGL(d,\Rbb)$ orbit is a positive configuration of $t$ flags.
\end{definition}


The choice of signs in Equation \ref{eqn:snakebasis} guarantees plus signs in Equations \ref{eqn:tail} and \ref{eqn:diamond} and, consequently, it implies the following.
 
\begin{cor} Fix a positive quadruple of flags $(E,F,G,H)$ and a non-zero vector in $(E^{d-1})^\perp$. 
\begin{itemize}
	\item[-] For any two snakes $\sigma$ and $\sigma'$, the upper triangular matrix $\mathcal M^{\sigma'}_{\sigma}(E,F,G)\in\GL(d,\Rbb)$ describing the snake bases change is totally nonnegative.
	\item[-] The diagonal matrix $\mathcal S(E,F,G,H)$ describing the shearing basis change is totally nonnegative.
	\item[-] There exist bases $\mathcal U=(u_i)$ and $\mathcal U'=(u_i')$ of $(\Rbb^d)^*$ such that 
	\begin{gather*}
	(E^{(d-i)})^\perp=\mathrm{Span}(u_1,\dots, u_i), \ (G^{(d-i)})^\perp=\mathrm{Span}(u_d,\dots, u_{d-i+1})\\
	(F^{(d-i)})^\perp=\mathrm{Span}(u'_1,\dots, u'_i), \ (H^{(d-i)})^\perp=\mathrm{Span}(u'_d,\dots, u'_{d-i+1})
	\end{gather*}
	and the matrix in the basis $\mathcal U$ of the element sending $\mathcal U$ to $\mathcal U'$ is totally nonnegative.
\end{itemize}
\end{cor}

\section{The building $\Bcal_d$}\label{sec:building}

\subsection{Axiomatic definition of $\Rbb$-Euclidean building}\label{ssec:axiomatic}

Let us recall the axiomatic definition of an $\Rbb$-Euclidean building associated to the general linear group. 

Consider the affine space
\[
\Abb^{d-1}=\{(x_1,x_2,\dots,x_d)\in \Rbb^d\colon x_1+x_2+\dots+x_d=1\}
\]
with underlying vector space
\[
\mathbb V^{d-1}=\{(x_1,x_2,\dots,x_d)\in \Rbb^d\colon x_1+x_2+\dots+x_d=0\}.
\]

Let $\mathfrak S_d$ denote the symmetric group on $d$ elements and let the \emph{affine Weyl group} be the semi-direct product $W_{\text{aff}}:=\mathfrak S_d\ltimes \mathbb V^{d-1}$. The symmetric group acts on $\Abb^{d-1}$ permuting the coordinates, and $\mathbb V^{d-1}$ acts on $\Abb^{d-1}$ by translations, therefore $W_{\text{aff}}$ acts on $\Abb^{d-1}$.

The standard inner product in $\Rbb^d$ induces a $\mathfrak S_d$-invariant inner product on $\Abb^{d-1}$ for which the elements of $W_{\text{aff}}$ are isometries.

The \emph{fundamental Weyl chamber} of $\Abb^{d-1}$ is the cone $\mathfrak C:=\{x\in\Abb^{d-1}\colon x_1\geq x_2\geq \dots \geq x_d\}$. A \emph{Weyl sector} $S$ is an image of $\mathfrak C$ via an element of $w\in W_{\text{aff}}$. 

\begin{definition}\label{def:building}
An \emph{$\Rbb$-Euclidean building modeled on $(\Abb^{d-1},W_{\text{aff}})$} is a set $\Bcal$ together with a family $\mathscr{A}$ of injective maps $f\colon \Abb^{d-1}\to \Bcal$ satisfying the following axioms:
\begin{itemize}
	\item[1.] if $f\in\mathscr A$, then $f \circ w\in\mathscr A$ for any element $w\in W_{\text{aff}}$;
	\item[2.] for any $f,f'\in\mathscr A$, the set $I:=(f^{-1}\circ f')(\Abb^{d-1})\subset\Abb^{d-1}$ is convex, closed and $(f^{-1}\circ f')_{\vert_I}$ is the restriction to $I$ of some $w\in W_{\text{aff}}$;
	\item[3.] Any two points $x$, $y$ belong to a common apartment;
	\item[4.] Any two Weyl sectors in $\Bcal$ contain Weyl subsectors contained in a common apartment;
\end{itemize}
Axioms 2 and 3 imply that the distance in $\Abb^{d-1}$ induces a distance in $\Bcal$.
\begin{itemize}
	\item[5.] For any point $x\in\Bcal$ and any $f\in\mathscr A$ such that $x\in f(\Abb^{d-1})$, there exists a retraction $r_{x,f}$ of $\Bcal$ onto $\Abb^{d-1}$ such that $r_{x,f}^{-1}(x)=x$ and $r_{x,f}$ decreases distances.
	\end{itemize}
\end{definition}
An element of $\mathscr A$ is called a \emph{marking}; the image of $\Abb^{d-1}$ under a marking is an \emph{apartment}. It follows from item 2. in Definition \ref{def:building} that any two markings of a given apartment differ by an element of the affine Weyl group.

\subsection{Asymptotic cones}

We will focus on a specific $\Rbb$-Euclidean building $\Bcal_d$ that admits an explicit model we describe in \S \ref{ssec:B_d}. We start by recalling some concepts from non-standard analysis. We refer to standard references \cite{Gro2, KapLe, KLe, Par1, DriW} for detailed discussions.

A \emph{non-principal ultrafilter} $\omega$ is a finitely additive measure on the natural numbers with values in $\{0,1\}$ and such that $\omega(S)=0$ whenever $S$ is finite. Given a sequence $(x_n)\subset\Rbb$ we say that  $x\in[-\infty, +\infty]$ is the {\em $\omega$-limit of $x_n$}, and we write $x:=\lim_\omega x_n$, if for any neighborhood $U$ of $x$ one has $x_n\in U$ for $\omega$-almost every $n$. Because $[-\infty,+\infty]$ is a compact and Hausdorff topological space, every sequence $(x_n)\subset \Rbb$ has a unique $\omega$-limit in $[-\infty,+\infty]$. 


\begin{notation} Let us fix once and for all a non-principal ultrafilter $\omega$ and a \emph{scaling sequence} $\lambda:=(\lambda_n)\subset\Rbb$ such that $\lambda_n\geq 1$ and $\lim_n\lambda_n= \infty$. 
\end{notation}

\begin{definition}\label{def:ascone} Let $(X,d,x_0)$ be a metric space with basepoint $x_0$. The \emph{asymptotic cone of $(X,d,x_0)$ with respect to the non-principal ultrafilter $\omega$ and the scaling sequence $\lambda$} is the set
\[
\mathcal C_{\omega,\lambda}(X,d,x_0)=\left\{(x_n)\in\prod_nX\colon \lim_\omega d(x_0,x_n)^{1/\lambda_n}<\infty\right\}_{\bigg/\mathlarger{\sim}}
\]
where $(x_n)\sim(y_n)$ if the $\omega$-limit $\lim_\omega d(x_n,y_n)^{1/\lambda_n}$ is zero.
\end{definition}

The \emph{ultralimit} $x$ of a sequence $(x_n)\in\prod_nX$ such that $\lim_\omega d(x_n,y_n)^{1/\lambda_n}<\infty$ is the equivalence class of $(x_n)$ in the asymptotic cone $\mathcal C_{\omega,\lambda}(X,d,x_0)$. The asymptotic cone is a complete metric space when equipped with the distance $d(x,y):=\lim_\omega d(x_n,y_n)^{1/\lambda_n}$ (Cf. \cite[\S 2]{KLe}).

Recall that a \emph{valuation} on a field $\mathbb K$ is an application $v\colon \mathbb K\to\Rbb\cup\{\infty\}$ such that for $x,y\in\mathbb K$
\begin{itemize}
	\item[-] $v(x)=\infty$ if and only if $x=0$;
	\item[-] $v(xy)=v(x)+v(y)$;
	\item[-] $v(x+y)\geq \min\{v(x),v(y)\}$ with equality whenever $v(x)\neq v(y)$.
\end{itemize}
Moreover, a valuation defines an associated \emph{absolute value} $\absv{x}:=e^{-v(x)}$ where, by convention, $e^{-\infty}=0$.

An example of an asymptotic cone is obtained by considering $\Rbb$ as a metric space with distance given by the absolute value and basepoint 0. It turns out that $\Fbb:=\mathcal C_{\omega,\lambda}(\Rbb,|\cdot|,0)$ is a field when equipped with the natural sum and multiplication of sequences (cf. \cite[p. 69]{Par3}). 

The field $\Fbb$ has a natural valuation given by
\begin{align*}
v\colon \ \mathbb F \ &\to\Rbb\cup\{\infty\}\\
x&\mapsto -\lim_\omega \left(\log \vert x_n\vert^{1/\lambda_n}\right)
\end{align*}
We embed $\Rbb$ in $\mathbb F$ via constant sequences, and we observe that $v(\Rbb-\{0\})=0$. 


\subsection{A concrete model for $\mathcal B_d$}\label{ssec:B_d}

\subsubsection{Ultrametric norms}

Let $V$ be a $d$-dimensional $\mathbb F$ vector space.

\begin{definition}\label{def:ultrametric norms}
An \emph{ultrametric norm} $\eta$ on $V$ is a function $\eta\colon V\to \Rbb$ such that for all $w,z\in V$ and all $x\in \Fbb$
\begin{itemize}
	\item[-] $\eta(w)=0$ if and only if $w=0$;
	\item[-] $\eta(xw)=\absv{x}\eta(w)$;
	\item[-] $\eta(w+z)\leq\max\{\eta(w),\eta(z)\}$.
\end{itemize}
\end{definition}
The absolute value $\absv{\cdot}$ on $\Fbb$ is an example of an ultrametric norm on $V=\Fbb$. 

Let $\Ecal=(e_1,e_2,\dots,e_d)$ be a basis of $V$. We say that the ultrametric norm $\eta$ is \emph{adapted to $\Ecal$} if for any $w=x_1e_1+x_2e_2+\dots+x_de_d\in V$
\[
\eta(w)=\max_{j=1,\dots,d} \absv{x_j}\eta(e_j).
\]
An ultrametric norm $\eta$ is \emph{adaptable} if there exists a basis $\Ecal$ of $V$ so that $\eta$ is adapted to $\Ecal$. Two ultrametric norms $\eta,\eta'$ are \emph{homothetic} if there exists $x\in\Fbb-\{0\}$ such that for every vector $w\in V$, $\eta(w)=\eta'(xw)$.

\begin{thm}[Parreau \cite{Par1}]\label{thm:normmodel} The set $\Bcal_d$ of homothety classes of adaptable ultrametric norms on the $d$-dimensional $\Fbb$-vector space $V$ is an $\Rbb$-Euclidean building modeled on $(\Abb^{d-1},W_{\text{aff}})$.
\end{thm}

The action of $g\in\GL(V)$ on an ultrametric norm $\eta$ is given by $g.\eta=\eta\circ g^{-1}$. Note that scalar matrices act by homothety on ultrametric norms, therefore the action of $\GL(V)$ descends to an action of $\PGL(V)$ on $\Bcal_d$. If $\eta$ is adapted to the basis $\Ecal$, then $g.\eta$ is adapted to $g\Ecal$. It is easy to see that $\PGL(V)$, and therefore $W_{\text{aff}}$, acts on $\Bcal_d$ via isometries.

\subsubsection{Apartments}\label{sssec:apts}

Any basis $\Ecal=(e_1,\dots,e_d)$ of the vector space $V$ determines a \emph{standard marking}
\begin{align*}
f_\Ecal\colon \ \ \ \Abb^{d-1}\  &\rightarrow \ \ \ \ \ \Bcal_d\\
\begin{pmatrix}x_1\\x_2\\\vdots\\ x_d\end{pmatrix}&\mapsto \begin{cases} [\eta]\colon \eta\text{ is adapted to }\Ecal,\\\eta(e_j)=e^{-x_j} 
\end{cases}
\end{align*}
\begin{remark}\label{rmk:actiononapts}
Note that the apartment $\Acal_{\Ecal}:=f_{\Ecal}(\Abb^{d-1})$ depends exclusively on the line decomposition $L_\Ecal$ defined as $(L_\Ecal)_i=\mathrm{Span}(e_i)$. The action of the affine Weyl group on $\Abb^{d-1}$ can be interpreted via the standard marking $f_\Ecal$ as an action on the set of bases that define the line decomposition $L_\Ecal$, or, equivalently, the set of markings of $\Acal_{\Ecal}$. More explictly, let $\sigma$ be a permutation in $\mathfrak S_d$ and denote by $\sigma\Ecal$ the basis $(e_{\sigma(1)},\dots, e_{\sigma(d)})$. Then,
\[
f_\Ecal(x_1,x_2,\dots,x_d)=f_{\sigma\Ecal}(x_{\sigma(1)},x_{\sigma(2)},\dots,x_{\sigma(d)}).
\]
Likewise, if $y_1,y_2,\dots, y_d\in\Fbb-\{0\}$, denote by $y\Ecal$ the basis $(y_1e_1,y_2e_2,\dots, y_de_d)$. Then,
\[
f_\Ecal(x_1,\dots,x_d)=f_{y\Ecal}(x_1+\tilde{y}_1,x_2+\tilde{y}_2,\dots,x_d+\tilde{y}_d).
\]
where $(\tilde{y}_1,\tilde y_2,\dots, \tilde{y}_d)$ is the unique vector in $\mathbb V^{d-1}$ such that $\tilde{y}_i-\tilde{y}_{i+1}=v(y_{i})-v(y_{i+1})$.
\end{remark}

\subsection{Intersection of apartments}\label{ssec:interofapts}

We outline a general algorithm to parametrize the intersection of apartments in $\Bcal_d$. We refer to \cite{Par1} for proofs. 

One of the main contributions of this paper is to show that total nonnegativity can be used to simplify this algorithm.

\begin{notation}
For the rest of this section, let $\Ecal=(e_1,\dots, e_d)$ and $\Ecal'=(e_1',\dots,e_d')$ be bases of $V$, $g\in\mathrm{GL}(V)$ be such that $g\Ecal=\Ecal'$ and $(g_{ij})_{1\leq i,j\leq d}$ be the matrix of $g$ in the basis $\Ecal$.
\end{notation}

\textbf{Step 1: }(Cf. \cite[Cor. 3.3]{Par1}) There exists $\overline{\sigma}\in \mathfrak S_d$ such that for all ultrametric norms $\eta$ adapted to both $\Ecal$ and $\Ecal'$
\[
\eta(e_j')=\absv{g_{\overline{\sigma}(j)j}}\eta(e_{\overline\sigma(j)}).
\]
Moreover, in this case
\[
v(\det g)=\min_{\sigma\in\mathfrak S_d} v(g_{\sigma(1)1}\cdots g_{\sigma(d)d})=v(g_{\overline{\sigma}(1)1}\cdots g_{\overline{\sigma}(d)d}).
\]

In other words, we can reorder the elements of $\Ecal$ so that the product of the diagonal entries has the same valuation as the determinant. Note that, in general, $v(\det g)$ is only greater or equal to $\min_{\sigma\in\mathfrak S_d} v(g_{\sigma(1)1}\cdots g_{\sigma(d)d})$.

Consider the apartments $\Acal=f_\Ecal(\Abb^{d-1})$ and $\Acal'=f_{\Ecal'}(\Abb^{d-1})$ and assume that the intersection of these two apartments is non-empty. This is equivalent to saying that there exists $\eta$ adapted to both $\Ecal$ and $\Ecal'$.

\textbf{Step 2: }(Cf. \cite[\S 3.4]{Par1})  Suppose the permutation $\overline \sigma$ of Step 1 is the identity and $g_{ii}=1$ for $i=1,2,\dots,d$. Then,
\[
\Acal\cap\Acal'=\{[\eta]\in\Acal\colon g.\eta=\eta\}.
\]

\textbf{Step 3: }(Cf. \cite[Prop. 3.5]{Par1}) Suppose $\absv{\det g}=1$. Then, the subset $\{[\eta]\in\Acal\colon g.\eta=\eta\}$ of the apartment $\Acal$ is the image under the standard marking $f_\Ecal$ of the set
\[
\{x\in\Abb^{d-1}\colon -v(g_{ij})\leq x_i-x_j\leq v(g_{ji})\text{ for }1\leq i<j\leq d\}.
\]

The next proposition follows by combining Steps 1, 2 and 3 above and it is used implicitly in \cite[\S 3.4]{Par1}.

\begin{prop}[Intersection of apartments]\label{prop:intersection} Consider bases $\Ecal=(e_1,\dots,e_d)$ and $\Ecal'=(e_1',\dots,e_d')$ of the $\Fbb$-vector space $V$. Let $g\in\mathrm{GL}(V)$ be such that $g\Ecal=\Ecal'$ and let $(g_{ij})_{1\leq i,j\leq d}$ be the matrix of $g$ in the basis $\Ecal$. Denote by $\Acal=f_\Ecal(\Abb^{d-1})$ and $\Acal'=f_{\Ecal'}(\Abb^{d-1})$ the apartments in $\Bcal_d$ defined via the standard markings and assume $\Acal\cap\Acal'\neq \emptyset$. Then, there exists a permutation $\overline\sigma\in W$ such that $\Acal\cap\Acal'$ is the image under the standard marking $f_\Ecal$ of the set
\begin{gather*}
\left\{x\in \Abb^{d-1}\colon -v\left(\frac{g_{\overline \sigma(i)j}}{g_{\overline\sigma(i)i}}\right)\leq x_{\overline\sigma(i)}-x_{\overline\sigma(j)}+v\left(\frac{g_{\overline\sigma(i)i}}{g_{\overline\sigma(j)j}}\right)\leq v\left(\frac{g_{\overline\sigma(j)i}}{g_{\overline\sigma(j)j}}\right)\right\}\\
=\left\{x\in \Abb^{d-1}\colon -v\left(\frac{g_{\overline \sigma(i)j}}{g_{\overline\sigma(j)j}}\right)\leq x_{\overline\sigma(i)}-x_{\overline\sigma(j)}\leq v\left(\frac{g_{\overline\sigma(j)i}}{g_{\overline\sigma(i)i}}\right)\right\}.
\end{gather*}
\end{prop}
\begin{proof} Let $\overline\sigma$ be as in Step 1 and identify it with the permutation matrix in $\mathrm{GL}(d,\Fbb)$ it defines with respect to the basis $\Ecal$. The matrix 
\[
\overline{g}=\text{diag}(1/g_{\overline\sigma(1)1},\dots,1/g_{\overline\sigma(d)d})\cdot\overline\sigma\cdot g
\]
satisfies the hypothesis of Step 3 with respect to the bases $\overline{\Ecal}=(g_{\overline\sigma(j)j}e_{\overline{\sigma}(j)})$ and $\Ecal'$. Therefore, the intersection $\Acal\cap\Acal'$ is the image under the marking $f_{\overline{\Ecal}}$ of
\[
\{x\in\Abb^{d-1}\colon -v(\overline g_{ij})\leq x_i-x_j\leq v(\overline g_{ji})\text{ for }1\leq i,j\leq d\}.
\]
On the other hand, as $\overline\Ecal=(g_{\overline\sigma(1)1},\dots,g_{\overline\sigma(d)d})\cdot\overline\sigma\cdot\Ecal$, by Remark \ref{rmk:actiononapts} we have
\[
f_\Ecal(x_1,\dots,x_d)=f_{\overline\Ecal}(x_{\overline\sigma(1)}+\tilde{g}_{\overline\sigma(1)1},\dots, x_{\overline\sigma(d)}+\tilde{g}_{\overline\sigma(d)d}),
\] 
where $(\tilde{g}_{\overline\sigma(1)1},\tilde g_{\overline\sigma(2)2},\dots, \tilde{g}_{\overline\sigma(d)d})$ is the unique vector in $\mathbb V^{d-1}$ such that $\tilde{g}_{\overline\sigma(i)i}-\tilde{g}_{\overline\sigma(i+1)i+1}=v(g_{\overline\sigma(i)i})-v(g_{\overline\sigma(i+1)i+1})$. We end the proof by observing that $\overline g_{ij}=g_{\overline{\sigma}(i)j}/g_{\overline{\sigma}(i)i}$.
\end{proof}

\subsection{Flags and Endomorphisms in $\Fbb^d$}\label{ssec:flagsendo}

We now collect a few properties of asymptotic cones that will be needed in what follows. Equip the vector space $V=\Fbb^d$ with the sup-norm
\[
\|x_1e_1+\dots+x_de_d\|_\omega=\max_i{\absv{x_i}}
\]
where $e_1,e_2.\dots, e_d$ denotes the standard basis of $\Fbb^d$.
\begin{prop}\label{prop:ultraV} The pointed normed vector space $(\Fbb^d,\|\cdot\|_\omega,0)$ is isomorphic to the asymptotic cone
\[
\mathcal C_{\omega,\lambda}(\Rbb^d,\|\cdot\|,0)
\]
for the standard Euclidean norm $\|\cdot\|$ on $\Rbb^d$.
\end{prop}
\begin{proof} Cf. \cite[Prop. 3.12]{Par2}.
\end{proof}

Proposition \ref{prop:ultraV} will allow us to study the asymptotic behavior of sequences of positive tuples of flags in $\Rbb^d$ in terms of the building $\Bcal_d$.

\begin{definition}\label{def:ultraW} Let $W_n$ be a sequence of $i$-dimensional vector subspaces in $\Rbb^d$. The $i$-dimensional subspace $W$ of $\Fbb^d$ is an \emph{ultralimit} for the sequence $W_n$ if the there exists 
\begin{itemize}
	\item[-] a sequence $(v_{1,n},v_{2,n},\dots, v_{i,n})$ of bases for $W_n$,
	\item[-] a basis $(v_1,v_2,\dots, v_i)$ of $W$
\end{itemize}
such that $\ulim v_{j,n}=v_j$ for all $j$.
\end{definition}

\begin{lem}\label{lem:ulimss} The sequence $W_n$ of $i$-dimensional vector subspaces in $\Rbb^d$ has a unique ultralimit $W$.
\end{lem}
\begin{proof} Existence of the ultralimit is obtained by choosing an orthonormal basis $(v_{1,n},v_{2,n},\dots, v_{i,n})$ for each $W_n$ and considering the ultralimits of these vectors. In fact, as each $v_{j,n}$ has constant norm equal to one, by definition of the asymptotic cone $v_{j,n}$ has a non-zero ultralimit $v_j$. We show by contradiction that the vectors $v_j$ are independent. Suppose there exists $1\leq l\leq i$ such that $v_l=\sum_{j\neq l} x_jv_j$ in $\Fbb^d$. There exist sequences of real numbers $(x_{j,n})$ such that $x_j=\ulim x_{j,n}\in\Fbb$. It follows that
\[
\left(1+\sum_jx^2_{j,n}\right)^{1/\lambda_n}=\left\|v_{l,n}-\sum_{j\neq l}x_{j,n}v_{j,n}\right\|^{1/\lambda_n}\geq 1
\]
has ultralimit equal to zero, which is a contradiction.

Let us now prove the uniqueness of the ultralimit $W$ of the sequence $W_n$. Suppose, by contradiction, that there exist two ultralimits $W$ and $W'$ for $W_n$ obtained by considering the sequences of bases $v_{j,n}$ and $v'_{j,n}$ of $W_n$ with ultralimits $v_j\in W$ and $v_j'\in W'$. Write $v'_{j,n}=\sum_k x_{k,j,n}v_{k,n}$. As we know that the ultralimits of $v_{j,n}$ and $v'_{j,n}$ are non-zero vectors in $\Fbb^d$, it follows that the ultralimits $x_{k,j}$ of the sequences $x_{k,j,n}$ are elements in $\Fbb$. In particular, we have that $v'_j=\sum_k x_{k,j}v_k$ and $v'_j$ belongs to $W$. As this holds for every $v_j'$, it follows that $W'\subset W$. We obtain the reverse inclusion $W'\supset W$ analogously, therefore $W=W'$ as needed.
\end{proof}

\begin{definition} Let $F_n$ be a sequence of flags in $\Rbb^d$. The flag $F$ in $\Fbb^d$ is the \emph{ultralimit} of the sequence $F_n$ of real flags if the there exist a sequence $(v_{1,n},v_{2,n},\dots, v_{d,n})$ of bases of $\Rbb^d$ and a basis $(v_1,v_2,\dots,v_d)$ of $\Fbb^d$ such that:
\begin{itemize}
	\item[-] for each $i$, the sequence of vectors $v_{i,n}$ in $\Rbb^d$ converges to the non-zero vector $v_i$ in $\Fbb^d$;
	\item[-] for each $n$, the sequence of $i$-dimensional vector subspaces $F_n^{(i)}=\mathrm{Span}(v_{1,n},v_{2,n},\dots, v_{i,n})$ converges to the vector subspace $F^{(i)}=\mathrm{Span}(v_1,v_2,\dots, v_i)$.
\end{itemize} 
\end{definition}

\begin{lem}\label{lem:ulimflag} Let $F_n$ be a sequence of flags in $\Rbb^d$. Then, there exists a unique flag $F$ in $\Fbb^d$ such that $F$ is the ultralimit of the sequence $F_n$.
\end{lem}
\begin{proof} The proof follows by applying Lemma \ref{lem:ulimss} to each sequence of $i$-dimensional subspaces $F_n^{(i)}$.
\end{proof}

The algebra $\mathrm{End}(V)$ of endomorphisms of $V=\Fbb^d$ can also be identified with an asymptotic cone. Observe that the norm $\|\cdot\|_\omega$ on $V=\Fbb^d$ induces an operator norm $N_\omega$ on $\mathrm{End}(V)$. 

\begin{prop}\label{prop:ulimend} The pointed normed algebra $(\mathrm{End}(V),N_\omega,\mathrm{Id})$ is isomorphic to the asymptotic cone
\[
\mathcal C_{\omega,\lambda}(\mathrm{End}(\Rbb^d),N,\mathrm{Id})
\]
where $N$ is the operator norm induced by the Euclidean norm $\|\cdot\|$ on $\Rbb^d$. 

Furthermore, let us identify $\mathrm{End}(V)$ and $\mathrm{End}(\Rbb^d)$ with the spaces of matrices $M(d,\Fbb)$ and $M(d,\Rbb)$ via the standard bases. Suppose that $M=(m_{ij})\in M(d,\Fbb)$ is the ultralimit of a sequence of matrices $M_n=((m_{ij})_n)\in M(d,\Rbb)$. Then, for every $i$ and $j$, we have that $m_{ij}=\ulim (m_{ij})_n$

Finally, the group $\GL(V)$ of invertible isomorphisms in $\mathrm{End}(V)$ is identified with the set of ultralimits of sequences $(g_n)\in\mathrm{End}(\Rbb^d)$ such that $g_n\in\GL(\Rbb^d)$ for $\omega$-almost every $n$ and $\lim_\omega N(g^{-1}_n)^{1/\lambda_n}<+\infty$.
\end{prop}
\begin{proof} Cf. \cite[Prop. 3.17, Cor. 3.18]{Par2} and \cite[Prop. 5.1]{Par4}.
\end{proof}

\section{Positivity in $\mathcal B_d$}\label{sec:posinBd}
%

Recall that we fixed a non-principal ultrafilter $\omega$ and a scaling sequence $\lambda=(\lambda_n)$. This allows us to consider the asymptotic cone $\Fbb$ of the real numbers $\Rbb$ with base point $0$ and distance given by the absolute value. Every element in $\Fbb$ is an equivalence class of sequences of real numbers. Therefore, the field $\mathbb F$ is naturally equipped with an order by setting
\[
[x_n]\geq [y_n] \text{ if } x_n\geq y_n\ \omega\text{-a.e.}.
\]
The set $\mathbb F_{\geq 0}=\{x\in\mathbb F\colon x\geq 0\}$ is a semifield with respect to the operations in $\mathbb F$ and it contains $\Rbb_{\geq 0}$. Set $\Fbb_{>0}:=\Fbb_{\geq 0}-\{0\}$. Total nonnegativity and total positivity can be defined naturally for elements in $\GL(d,\Fbb)$ as follows.

\begin{definition}\label{def:gentotpos} An element $M\in\GL(d,\Fbb)$ is \emph{totally positive} if all of its minors belong to $\Fbb_{> 0}$. The matrix $M\in\GL(d,\Fbb)$ is \emph{totally nonnegative} if all of its minors are in $\Fbb_{\geq 0}$.
\end{definition}

\subsection{Positivity and intersections}\label{ssec:positivityandintersections}

The main goal of this section is to show how total nonnegativity can be used to simplify the problem of parametrizing the intersection of two apartments $\mathcal A$ and $\Acal'$ in the $\Rbb$-Euclidean building $\Bcal_d$. More precisely, assume that $\Acal\cap\Acal'$ is non-empty. Proposition \ref{prop:intersection} states that, in general, the intersection of these two apartments is described by $d(d-1)$ inequalities. Corollary \ref{cor:posintersection} below shows that $2(d-1)$ inequalities suffice when $\Acal$ and $\Acal'$ are related by a totally nonnegative matrix. 

We will need the following technical lemmas.
\begin{lem}\label{lem:posvaluation} For any $x,y$ in $\Fbb_{\geq 0}$, we have that $v(x+y)=\min\{v(x),v(y)\}$.
\end{lem}
\begin{proof} Cf. \cite[Prop. 3.2.1.]{Par2}.
\end{proof}

\begin{lem}\label{lem:wellbehaved} Suppose $y,x-y$ are in $\Fbb_{\geq 0}$ and $y$ is different from zero. Then, $v(xy^{-1})\leq 0$.
\end{lem}
\begin{proof} As $\Fbb_{\geq 0}$ is a semifield and $y\neq 0$, we know that $xy^{-1}-1=(x-y)y^{-1}\in\Fbb_{\geq 0}$. Thus,
\begin{align*}
v(xy^{-1})&=v((xy^{-1}-1)+1)\\
&=\min\{v(xy^{-1}-1),0\}\leq 0.
\end{align*}
where the second equality follows from Lemma \ref{lem:posvaluation}, observing that $1\in\Fbb_{\geq 0}$ and $v(1)=0$.
\end{proof}

\begin{prop}\label{prop:positiveintersection} Let $M=(m_{ij})$ be a matrix in $\GL(d,\Fbb)$ with $\absv{\det M}=1$ and $m_{ii}=1$. Consider the sets
\begin{gather*}
\mathcal{I}_M=\{x\in \Abb^{d-1}\colon -v(m_{ij})\leq x_i-x_j\leq v(m_{ji}) \text{ for } 1\leq i<j\leq d\};\\
\mathcal{I}^+_M=\{x\in \Abb^{d-1}\colon -v(m_{i,i+1})\leq x_i-x_{i+1}\leq v(m_{i+1,i})\text{ for } 1\leq i<d\}.
\end{gather*}
If $M$ is totally nonnegative, then $\mathcal I_M=\mathcal I^+_M$.
\end{prop}
\begin{proof} It is clear that $\mathcal I_M\subseteq \mathcal I_M^+$. We show that if $x\in\mathcal I_M^+$, then $x_i-x_{i+k}\leq v(m_{i+k,i})$ by induction on $k\geq 2$. We omit the proof of the inequality $x_i-x_j\geq -v(m_{ij})$ as it is very similar. 

For $k=2$, we want to show $x_i-x_{i+2}\leq v(m_{i+2,i})$. Namely, we focus on the sub-matrix
\[
\begin{pmatrix}
 1 & \star & \star \\
 m_{i+1,i} & 1 &\star \\
 m_{i+2,i} & m_{i+2,i+1} & 1
\end{pmatrix}
\]
If $m_{i+2,i}=0$ there is nothing prove as $v(m_{i+2,i})=\infty$. Assume $m_{i+2,i}\neq 0$. Total nonnegativity of $M$ implies that $m_{i+2,i}$ and $m_{i+1,i}m_{i+2,i+1}-m_{i+2,i}$ are in $\Fbb_{\geq 0}$. This implies that $m_{i+1,i}$ and $m_{i+2,i+1}$ are non-zero. Therefore, the valuations of $m_{i+1,i}$ and $m_{i+2,i+1}$ are finite. We apply Lemma \ref{lem:wellbehaved} with $x=m_{i+1,i}m_{i+2,i+1}$ and $y=m_{i+2,i}$ to obtain
\begin{gather*}
v\left(\frac{m_{i+1,i}m_{i+2,i+1}}{m_{i+2,i}}\right)\leq 0\ \ \  \Longleftrightarrow \ \ \ v(m_{i+1,i})+v(m_{i+2,i+1})\leq v(m_{i+2,i}).
\end{gather*}
Thus, if $x\in \mathcal{I}^+_M$ we have
\[
x_i-x_{i+2}=x_i-x_{i+1}+x_{i+1}-x_{i+2}\leq v(m_{i+1,i})+v(m_{i+2,i+1})\leq v(m_{i+2,i})
\]
which proves the base case for the induction. Assume that for $x\in \mathcal I_M^+$ we know that $x_i-x_{i+l}\leq v(m_{i+l,i})$ whenever $l<k$. If $m_{i+k,i}=0$, the inequality $x_i-x_{i+k}\leq v(m_{i+k,i})=\infty$ is obvious. Thus, let us assume $m_{i+k,i}\neq 0$. We obtain the desired inequality
\[
x_i-x_{i+k}=x_i-x_{i+1}+x_{i+1}-x_{i+k}\leq v(m_{i+1,i})+v(m_{i+k,i+1})\leq v(m_{i+k,i})
\]
by using the induction hypothesis for the inequality $x_{i+1}-x_{i+k}\leq v(m_{i+k,i+1})$ and applying Lemma \ref{lem:wellbehaved} with $x=m_{i+1,i}m_{i+k,i+1}$ and $y=m_{i+k,i}$.
\end{proof}

We specialize Proposition \ref{prop:positiveintersection} to the case of upper triangular matrices for future reference. 

\begin{cor}\label{cor:uppertriangular} Let $M=(m_{ij})$ be an upper triangular matrix in $\GL(d,\Fbb)$ and consider the sets:
\begin{gather*}
\mathcal{I}_M=\left\{x\in \Abb^{d-1}\colon x_i-x_j+v\left(\frac{m_{ii}}{m_{jj}}\right)\geq -v\left(\frac{m_{ij}}{m_{ii}}\right) \text{ for } 1\leq i<j\leq d\right\};\\
\mathcal{I}^+_M=\left\{x\in \Abb^{d-1}\colon  x_i-x_{i+1}+v\left(\frac{m_{ii}}{m_{i+1,i+1}}\right)\geq -v\left(\frac{m_{i,i+1}}{m_{ii}}\right)\text{ for } 1\leq i<d\right\}.
\end{gather*}
If $M$ is totally nonnegative, then $\mathcal I_M=\mathcal I^+_M$.
\end{cor}
\begin{proof} As the determinant of $M$ is non-zero, we can multiply $M$ by the totally nonnegative diagonal matrix $S=\mathrm{diag}(1/m_{11},\dots,1/m_{dd})$. The matrix $M'=SM=(m_{ij}')$ is totally nonnegative, $\absv{\det M'}=1$ and its diagonal entries are equal to 1. Therefore, we conclude by applying Proposition \ref{prop:positiveintersection} to $M'$ and performing an easy algebraic manipulation.
\end{proof}

\begin{remark}
In the statement of Corollary \ref{cor:uppertriangular}, the expressions for the inequalities defining the sets $\mathcal{I}_M$ and $\mathcal{I}^+_M$ can be simplified by subtracting $v(m_{ii})$ on both sides. However, we wish to not do so as these two terms play different r\^oles when considering intersections of apartments in \S \ref{sec:main}. 
\end{remark}

\begin{cor}\label{cor:posintersection} Let $\Acal$ and $\Acal'$ be apartments in $\Bcal_d$. Assume that there exist bases $\Ecal$ and $\Ecal'$ of the $\Fbb$-vector space $V$ such that
\begin{itemize}
	\item[-] $\Acal=f_{\Ecal}(\Abb^{d-1})$ and $\Acal'=f_{\Ecal'}(\Abb^{d-1})$ where $f_{\Ecal}$ and $f_{\Ecal'}$ are the standard marking of the bases $\Ecal$ and $\Ecal'$, respectively;
	\item[-] the matrix $(g_{ij})$ in the basis $\Ecal$ corresponding to the group element $g\in\GL(V)$ such that $g\Ecal=\Ecal'$ is totally nonnegative, $g_{ii}=1$ and $v(\det g)=0$.
\end{itemize}
Then,
\[
f_{\Ecal}^{-1}(\Acal\cap\Acal')=\{x\in \Abb^{d-1}\colon -v(g_{i,i+1})\leq x_i-x_{i+1}\leq v(g_{i+1,i})\text{ for } 1\leq i<d\}.
\]
\end{cor}
\begin{proof} This is an immediate consequence of Proposition \ref{prop:intersection} and Proposition \ref{prop:positiveintersection}.
\end{proof}

\begin{example} Note that Proposition \ref{prop:positiveintersection} fails if $M$ is not totally nonnegative. Consider the sequence $(e^{\lambda_n})\subset \Rbb$ and observe that this defines a non-zero element in $\Fbb$ with valuation
\[
v([e^{\lambda_n}])=-\lim_\omega \frac{1}{\lambda_n}\log e^{\lambda_n}=-1
\]
The matrix $M=
\begin{pmatrix}
1 & 1 & [e^{\lambda_n}]\\
0 & 1 & 1\\
0 & 0  & 1
\end{pmatrix}
$
is not totally nonnegative as it has a minor equal to $[1-e^{\lambda_n}]<0$.

Using the notations introduced in Proposition \ref{prop:positiveintersection}, the set $\mathcal I_M$ is defined by the inequalities
\[
x_1-x_2\geq 0, \ x_2-x_3\geq 0, \ x_1-x_3\geq -v([e^{\lambda_n}])=1
\]
and it is properly contained in $\mathcal I^+_M$, which is defined by the inequalities $x_1-x_2\geq 0$, and $x_2-x_3\geq 0$.
\end{example}

\subsection{Positivity of configurations of flags in $\Fbb^d$}\label{ssec:positivitysections}

In \S \ref{ssec:flagsendo}, we described how a sequence of tuple of flags in $\Rbb^d$ defines a tuple of flags in $\Fbb^d$. Fix an oriented triangulation $\mathscr T$ of the regular convex polygon with $t$ vertices $\mathscr P_t$. In particular, for any sequence of tuples of flags we obtain corresponding Fock-Goncharov parameters as described in \S \ref{ssec:positivity}.

\begin{definition}\label{def:ultrapositive} The ultralimit $(F_1,F_2,\dots, F_t)$ in $\Fbb^d$ of a sequence of $t$ real flags $(F_{1,n},F_{2,n},\dots, F_{t,n})$ is \emph{positive} if
\begin{enumerate}
	\item the tuple of flags $(F_1,F_2,\dots, F_t)$ has the maximum span property,
	\item the sequences of triple and double ratios $X_{a,b,c}(F_{i,n}, F_{j,n}, F_{k,n})$ and $Z_s(F_{i,n},F_{j,n},F_{k,n},F_{l,n})$ with respect to the oriented triangulation $\mathscr T$ are such that
	\begin{gather*}
	0<\ulim X_{a,b,c}(F_{i,n}, F_{j,n}, F_{k,n})<\infty \text{ and } 0<\ulim Z_s(F_{i,n}, F_{j,n}, F_{k,n},F_{l,n})<\infty
	\end{gather*}
	for all $(a,b,c)\in\Theta^\circ_d$ and $s=1,2,\dots, d-1$.
\end{enumerate}
\end{definition}

\begin{remark}\label{rmk:posultra} Observe that, in Definition \ref{def:ultrapositive}, the maximum span property for the tuple $(F_1,F_2,\dots, F_t)$ of flags in $\mathbb F^d$ is independent on the choice of oriented triangulation. The positivity property is also independent on the choice of triangulation $\mathscr T$ thanks to \cite[\S 10]{FG1}. In fact, given any other oriented triangulation $\mathscr T'$ of the regular convex polygon with $t$ vertices $\mathscr P_t$, the sequences of Fock-Goncharov coordinates for $\mathscr T'$ can be expressed as a ratio of subtraction-free polynomials of the Fock-Goncharov coordinates with respect to the triangulation $\mathscr T$. It follows that the positivity and finiteness of the ultralimit of the coordinates is preserved by a change of triangulation.
\end{remark}

The following example illustrates how the maximum span property in Definition \ref{def:ultrapositive} is not implied by the positivity of the ultralimits of the Fock-Goncharov coordinates.

\begin{example}\label{ex:notransv} Consider the sequence of four flags $(E_n,F_n,G_n,H_n)$ in $\Rbb^2$ such that
\begin{align*}
E_n^{(1)}=\mathrm{Span}\begin{pmatrix} 1\\ 0\end{pmatrix},&\ G_n^{(1)}=\mathrm{Span}\begin{pmatrix} e^{\lambda_n^2} \\ 1\end{pmatrix},\\
F_n^{(1)}=\mathrm{Span}\begin{pmatrix} -1+e^{\lambda_n^2}\\ 1\end{pmatrix},&\ H_n^{(1)}=\mathrm{Span}\begin{pmatrix} 2+e^{\lambda_n^2} \\ 1\end{pmatrix},
\end{align*}
An easy computation shows that the sequence of double ratios of these four lines is constant equal to two. In particular, it is positive. However,
\[
\ulim E_n^{(1)}=\ulim F_n^{(1)}=\ulim G_n^{(1)}=\ulim H_n^{(1)}=\mathrm{Span}\begin{pmatrix} 1\\ 0\end{pmatrix}.
\]
\end{example} 

The following lemma gives a sufficient criterion for positivity of the ultralimit of a sequence of positive tuples of flags.

\begin{lem} Consider the ultralimit $(F_1,F_2,\dots, F_t)$ of a sequence of tuples of flags $(F_{1,n},F_{2,n},\dots, F_{t,n})$. Assume that there exists $1\leq i,j,k\leq t$ such that for all $a,b=0,1,\dots, d$ and $c=0,1$ we have
\[
\dim\left(F_i^{(a)}+F_j^{(b)}+F_k^{(c)}\right)=\min\{a+b+c,d\}.
\]
Assume that the Fock-Goncharov invariants of $(F_1,F_2,\dots, F_t)$ have finite positive ultralimits. Then, the ultralimit $(F_1,F_2,\dots, F_t)$ is positive.
\end{lem}
\begin{proof} Without loss of generality, let us assume $(i,j,k)=(1,2,3)$. Let $\mathscr T$ be an ideal triangulation of $\mathscr P_t$ such that the vertices labeled by $1$,$2$,$3$ and $4$ form a quadrilateral with diagonal labeled by the vertices $(v_1,v_3)$. In dimension $d=3$, this lemma is a consequence of \cite[Prop. 5.5]{Par4}. For $d>3$ one uses the following standard observation. For any triple of flags $(E_n,F_n,G_n)$ with the maximum span property, and for any $(a,b,c)\in\Theta^\circ_d$, the quotient flags $(\overline{E}_n,\overline{F}_n,\overline{G}_n)$ in $\Rbb^d/(E_n^{(a-1)}\oplus F_n^{(b-1)}\oplus G_n^{(c-1)})\cong \Rbb^3$ has the maximum span property. Moreover, it is easy to check that
\[
X_{a,b,c}(E_n,F_n,G_n)=X_{1,1,1}(\overline{E}_n,\overline{F}_n,\overline{G}_n).
\] 
Therefore, choose an index $(a,b,c)\in\Theta^\circ_d$ with $c=1$. Then, for every $n$, the quotient flag $\overline{F}_{3,n}$ is simply the flag with line $F_{3,n}^{(1)}$ and plane $F_{3,n}^{(2)}$. Letting $a$ and $b$ vary and applying \cite[Prop. 5.5]{Par4}, we have that $F_{3}^{(2)}$ is such that
\begin{equation}\label{eqn:lemgener}
\dim\left(F_1^{(a)}+F_2^{(b)}+F_3^{(2)}\right)=\min\{a+b+2,d\}
\end{equation}
Iterating this argument as we let $c$ vary between $2$ and $d-2$, we have that the limiting triple $(F_1,F_2,F_3)$ satisfies the maximum span property. 

A similar argument can be used to prove the maximum span property for quadruples of flags. In fact, whenever we have a sequence of maximum span quadruple of flags $(E_n,F_n,G_n,H_n)$, the quotient of $\Rbb^d$ by the subspace $E_n^{(i-1)}\oplus G_n^{(d-i-2)}$ is three-dimensional and defines a sequence of quadruples of flags $(\overline{E}_n,\overline{F}_n,\overline{G}_n,\overline{H}_n)$ such that 
\begin{gather*}
Z_i(E_n,F_n,G_n,H_n)=Z_1(\overline{E}_n,\overline{F}_n,\overline{G}_n, \overline{H}_n),\\
Z_{i+1}(E_n,F_n,G_n,H_n)=Z_2(\overline{E}_n,\overline{F}_n,\overline{G}_n, \overline{H}_n).
\end{gather*}
Therefore, consider the sequence of positive quadruples $(F_{1,n},F_{2,n}, F_{3,n}, F_{4,n})$. Applying \cite[Prop 5.5]{Par4} to the quadruples $(\overline{F}_{1,n},\overline{F}_{2,n},\overline{F}_{3,n},\overline{F}_{4,n})$ as we let $i$ vary between 1 and $d-2$, we obtain that 
\[
\dim\left(F_1^{(a)}+F_3^{(b)}+F_4^{(1)}\right)=\min\{a+b+1,d\}.
\] 
In summary, we showed that if the flags $(F_1,F_2,F_3)$ in $\Fbb^d$ satisfy Equation \ref{eqn:lemgener} and we have positivity of the ultralimits of the Fock-Gonchaorv coordinates of the quadruple $(F_1,F_2,F_3,F_4)$, then $(F_1,F_2,F_3)$ satisfies the maximum span property and $(F_1,F_3,F_4)$ satisfies Equation \ref{eqn:lemgener}. This finishes the proof as we can now iterate this procedure.
\end{proof}

Finally, recall that given a triple of flags $(E,F,G)$ in $\Rbb^d$, a snake $\sigma$ in $\Theta^\perp_d$ defines a projective basis for the space $(\Rbb^d)^*$. The following lemma states that snake bases are well behaved when we consider positive ultralimits of sequences of positive triples of flags.

\begin{lem}\label{lem:assnakes} Let $(E_n,F_n,G_n)$ be a sequence of positive triples of flags whose ultralimit $(E,F,G)$ is positive. Let $(u_{i,n})$ be the corresponding sequence of snake bases for the snake $\sigma$. Up to rescaling, assume that the sequence $(u_{1,n})$ of non-zero vectors in $(E_n^{(d-1)})^\perp$ is such that
\[
\ulim u_{1,n}=u_1\in\Fbb^d-\{0\}.
\]
Then, the ultralimit of $u_{i,n}$ is a non-zero vector in $\Fbb^d$ for every $i=1,2,\dots, d$.
\end{lem}
\begin{proof} Fix a snake $\sigma$ and denote by $(L_{1,n}, L_{2,n},\dots, L_{d,n})$ the sequence of line decompositions it defines via the triples of flags $(E_n,F_n,G_n)$ as described in \S \ref{ssec:snakes}. The result follows from the normalization for the vectors $u_{i,n}$ given in Lemma \ref{lem:snakebase} and by the maximum span property for the triple of flags $(E,F,G)$ in $\Fbb^d$. In fact, by construction of the snake basis, given $u_{i-1,n}\in L_{i-1,n}$, the vector $u_{i,n}\in L_{i,n}$ is defined recursively, up to a sign, as one of the two vectors $u'_{i,n}\in L'_i$ and $u''_{i,n}\in L''_i$ satisfying the equality $u_{i-1,n}+u'_{i,n}+u''_{i,n}=0$. On the other hand, the maximum span property implies that there exist sequences of non-zero real numbers $a_n'$ and $a_n''$ such that the ultralimits $\ulim a_n'u_{i,n}'$ and $\ulim a''_nu''_{i,n}$ are non-zero vectors $v_i'$ and $v_i''$ in $(\Fbb^d)^*$. If $a'=(a_n')$ and $a''=(a_n'')$ are elements in $\Fbb-\{0\}$, it follows that $\ulim u_{i,n}'=v_i'/a'$ and $\ulim u_{i,n}''=v_i''/a''$ are non-zero vectors in $(\Fbb^d)^*$. Therefore, we want to show that this has to be the case. This follows by writing
\[
u_{i-1,n}+\frac{1}{a'_n}(a'_nu'_{i,n})+\frac{1}{a''_n}(a''_nu''_{i,n})=0
\]
and observing that if $\lim_\omega |a'_n|^{1/\lambda_n}=0$ or $+\infty$, then $u_{i-1}\in\mathrm{Span}(v_i'')$ or $v_i'\in\mathrm{Span}(v_i'')$, respectively. In either case, this contradicts the maximum span property of the triple of flags $(E,F,G)$.
\end{proof}

\section{Positive intersections of flag apartments}\label{sec:main}

In this section we collect the proofs of our main results: Theorems \ref{intro:triple}, \ref{intro:quadruple} and \ref{intro:mono} from the introduction. Our main tool is Proposition \ref{prop:positiveintersection}, which we use to describe the geometry of a preferred collection of apartments in the $\Rbb$-Euclidean building $\Bcal_d$.

\subsection{Monotonicity for positive configurations of flags}\label{se:Kmono}

Consider a sequence of positive tuples of flags $(F_{1,n},F_{2,n},\dots, F_{t,n})$ in $\Rbb^d$. If the ultralimit $(F_1,F_2,\dots, F_t)$ is positive, any pair of flags $(F_i,F_j)$ defines a line decomposition of the $d$-dimensional vector space $V=(\Fbb^d)^*$. It follows from \S \ref{sssec:apts} that such a line decomposition determines an apartment $\Acal_{ij}$ in the $\Rbb$-Euclidean building $\Bcal_d$. Recall from the introduction that given three apartments $\Acal_{i_1j_1}$, $\Acal_{i_2j_2}$ and $\Acal_{i_3j_3}$ we say that $\Acal_{i_2j_2}$ \emph{combinatorially separates} $\Acal_{i_1j_1}$ and $\Acal_{i_3j_3}$ if, up to a cyclic permutation of the indices of the tuple of flags $(F_1,F_2,\dots, F_t)$, we have
\[
1\leq i_1\leq i_2\leq i_3<j_3\leq j_2\leq j_1\leq t.
\]

\begin{thm}[Theorem \ref{intro:mono}]\label{thm:monotone} Consider a sequence $(F_{1,n},F_{2,n},\dots, F_{t,n})$ of $t$ positive flags in $\Rbb^d$ with positive ultralimit $(F_1,F_2,\dots, F_t)$. Consider apartments $\Acal_1$, $\Acal_2$ and $\Acal_3$ defined via the line decompositions associated to pairs of flags $(F_{i_1},F_{j_1})$, $(F_{i_2},F_{j_2})$, and $(F_{i_3},F_{j_3})$, respectively. If the apartment $\Acal_2$ combinatorially separates $\Acal_1$ and $\Acal_3$, then
\[
\Acal_1\cap\Acal_3=\Acal_1\cap\Acal_2\cap\Acal_3.
\]
\end{thm}
\begin{proof} If $\Acal_1\cap\Acal_3=\emptyset$, the result is trivial. Therefore, we assume that the intersection $\Acal_1\cap\Acal_3$ is non-empty. Moreover, it suffices to show $\Acal_1\cap\Acal_3\subseteq \Acal_1\cap\Acal_2$. In fact, it then follows that $\Acal_1\cap\Acal_3\subseteq \Acal_1\cap\Acal_2\cap \Acal_3$ and the reverse inclusion is obvious. We subdivide the proof into three cases. The first two cases will need the following technical lemma.

\begin{lem}\label{lem:monolem} Let $A=(a_{ij})$, $B=(b_{ij})$ be totally nonnegative matrices in $\GL(d,\Fbb)$ such that $A=(a_{ij})$ is upper triangular and $B$ is a triangular matrix (upper or lower). Denote by $C=(c_{ij})$ the product $AB$ and assume that
\begin{equation}\label{eq:vdet}
v(\det C)=\min_{\sigma \in \mathfrak S_d}v(c_{\sigma(1)1}\dots c_{\sigma(d)d}).
\end{equation}
Then, $v(\det C)=v(c_{11}\dots c_{dd})$ and
\begin{equation}\label{eqn:monot}
v\Big(\frac{c_{i,i+1}}{c_{i+1,i+1}}\Big)\leq v\Big(\frac{a_{i,i+1}}{a_{i+1,i+1}}\Big).
\end{equation}
\end{lem}
\begin{proof} As $A$ and $B$ are totally nonnegative, we have the following implications 
\begin{align*}
c_{i,i+1}=a_{i,i+1}b_{i+1,i+1}+\sum_{j\neq i}a_{ij}b_{j,i+1} &\Longrightarrow v(c_{i,i+1})\leq v(a_{i,i+1}b_{i+1,i+1}),\\
c_{ii}=a_{ii}b_{ii}+\sum_{j\neq i}a_{ij}b_{ji}&\Longrightarrow v(c_{ii})\leq v(a_{ii}b_{ii}).
\end{align*}
Moreover, as $A$ and $B$ are triangular matrices in $\GL(d,\Fbb)$, we have that the valuations $v(a_{ii})$ and $v(b_{ii})$ are finite. Let us prove the equality $v(c_{ii})=v(a_{ii}b_{ii})$ for all $i$. We have \begin{align*}
v(a_{11}\dots a_{dd}b_{11}\dots b_{dd})&=v(\det C)\\
&\leq v(c_{11}\dots c_{dd})\\
&\leq v(a_{11}b_{11}\dots a_{dd}b_{dd})
\end{align*}
where the first inequality follows from Equation \ref{eq:vdet}. Therefore, $v(\det C)=v(c_{11}\dots c_{dd})$ and 
\begin{align*}
v(c_{i,i+1})-v(c_{i+1,i+1})&\leq v(a_{i,i+1})+v(b_{i+1,i+1}) - v(a_{i+1,i+1})-v(b_{i+1,i+1})\\
&=v(a_{i,i+1})-v(a_{i+1,i+1})
\end{align*}
which is equivalent to Equation \ref{eqn:monot}.
\end{proof}

It follows from Proposition \ref{prop:ulimend} and Lemma \ref{lem:assnakes} that there exist bases $\Ecal_i$, $i=1,2,3$ of $V$ such that 
\begin{itemize}
	\item[-] the apartment $\Acal_i$ is the image of $\Abb^{d-1}$ via the standard marking of the basis $\Ecal_i$,
	\item[-] for $i<j$, if $g_{ij}\in\GL(V)$ is the element such that $g_{ij}\Ecal_i=\Ecal_j$, then the corresponding matrix in the basis $\Ecal_i$ is totally nonnegative. 
\end{itemize}
Let $A$, $B$ and $C$ be the totally nonnegative matrices corresponding to the elements $g_{12}$, $g_{23}$ and $g_{13}$, respectively. Observe that $C=(ABA^{-1})A=AB$. Moreover, as $\Acal_1\cap\Acal_3\neq \emptyset$, Step 1 in \S \ref{ssec:interofapts} implies that the determinant of $C$ satisfies Equation \ref{eq:vdet}.

\textbf{Case 1}. Assume $F_{i_1}=F_{i_2}=F_{i_3}$. It follows that $A$ and $B$ are upper triangular as they need to preserve the flag $F_{i_k}$. In particular, the determinant of $C$ is $\prod_i c_{ii}=\prod_i a_{ii}b_{ii}$. By Corollary \ref{cor:uppertriangular} we know that 
\begin{gather*}
\Acal_1\cap\Acal_2= f_{\Ecal_1}\left(\left\{x\in\Abb^{d-1}\colon x_{i}-x_{i+1}\geq - v\left(\frac{a_{i,i+1}}{a_{i+1,i+1}}\right)\right\}\right)\\
\Acal_1\cap\Acal_3= f_{\Ecal_1}\left(\left\{x\in\Abb^{d-1}\colon x_{i}-x_{i+1}\geq - v\left(\frac{c_{i,i+1}}{c_{i+1,i+1}}\right)\right\}\right).
\end{gather*}
Therefore, Lemma \ref{lem:monolem} implies that $\Acal_1\cap\Acal_3\subseteq \Acal_1\cap\Acal_2$.

\textbf{Case 2}. Assume $F_{i_1}=F_{i_2}$ and $F_{j_2}=F_{j_3}$. Therefore, $A$ is upper triangular and $B$ is lower triangular. In particular,
\[
\Acal_1\cap\Acal_2= f_{\Ecal_1}\left(\left\{x\in\Abb^{d-1}\colon x_{i}-x_{i+1}\geq - v\left(\frac{a_{i,i+1}}{a_{i+1,i+1}}\right)\right\}\right)
\]
On the other hand, by the second part of Lemma \ref{lem:monolem} and Proposition \ref{prop:positiveintersection} we have
\[
\Acal_1\cap\Acal_3= f_{\Ecal_1}\left(\left\{x\in\Abb^{d-1}\colon -v\left(\frac{c_{i,i+1}}{c_{i+1,i+1}}\right)\leq x_{i}-x_{i+1}\leq v\left(\frac{c_{i+1,i}}{c_{ii}}\right)\right\}\right).
\]
Once again, Lemma \ref{lem:monolem} implies that $\Acal_1\cap\Acal_3\subseteq \Acal_1\cap\Acal_2$.

\textbf{General Case.} Recall that we denote by $\Acal_{ij}$ the apartment defined by the flags $F_i$ and $F_j$ so that $\Acal_k=\Acal_{i_kj_k}$. Using Case 2, we have
\[
\Acal_{1}\cap\Acal_{3}= \Acal_{1}\cap\Acal_{i_1,j_3}\cap \Acal_{3}.
\]
On the other hand, the previous cases imply the following inclusions
\begin{equation}\label{eq:mono}
\begin{gathered}
\Acal_{1}\cap\Acal_{i_1,j_3}\subseteq\Acal_{1}\cap\Acal_{i_1,j_2},\\
\Acal_{i_1,j_3}\cap\Acal_{3}\subseteq\Acal_{i_1,j_3}\cap\Acal_{i_2,j_3},\\
\Acal_{i_1,j_2}\cap \Acal_{i_2,j_3}\subseteq \Acal_{i_1,j_2}\cap\Acal_2.
\end{gathered}
\end{equation}
Therefore, it follows that
\begin{align*}
\Acal_1\cap\Acal_3&=\Acal_{1}\cap\Acal_{i_1,j_3}\cap \Acal_{3}\\
&\subseteq \Acal_1\cap \Acal_{i_1,j_2}\cap \Acal_{i_2,j_3}\\
&\subseteq \Acal_1\cap\Acal_2
\end{align*}
where the first inclusion follows from the first two lines in Equation \ref{eq:mono} and the second inclusion follows from the last line in Equation \ref{eq:mono}.\end{proof}

\subsection{Ultralimits of positive triples and intersection of apartments}\label{sec:snakesapart}

For the rest of this section, fix a sequence of positive triples of flags $(E_n,F_n,G_n)$ in $\Rbb^d$ with positive ultralimit the triple $(E,F,G)$ of flags in $\Fbb^d$. We ease notation by setting   
\[
X_{a,b,c}:=\ulim X_{a,b,c}(E_n,F_n,G_n),
\]
which, by hypothesis, is positive in the field $\Fbb$.

As the triple $(E,F,G)$ has the maximum span property, the choice of a snake $\sigma$ in $\Theta_d^\perp$ determines a line decomposition $\mathcal L_\sigma$ of $V=(\Fbb^d)^*$ and a corresponding apartment $\mathcal A_\sigma$ in the $\Rbb$-Euclidean building $\Bcal_d$.

\begin{lem}\label{lem:snakeapt} Let $\Acal_\sigma$ and $\Acal_{\sigma'}$ be apartments associated to snakes $\sigma$ and $\sigma'$. Then, $\Acal_\sigma\cap \Acal_{\sigma'}\neq\emptyset$.
\end{lem}
\begin{proof} Any snake determines a line decomposition $(L_1,L_2,\dots,L_d)$ of $V$ such that for every $i$, 
\[
L_1\oplus L_2\oplus\dots\oplus L_i=(E^{(d-i)})^\perp.
\]
It then follows from \cite[Prop. 3.8]{Par1} that given any two snakes $\sigma$ and $\sigma'$, the corresponding apartment $\Acal_\sigma$ and $\Acal_{\sigma'}$ intersect in, at least, a Weyl sector.
\end{proof}

As a consequence of Lemma \ref{lem:assnakes}, every snake $\sigma$ determines a projective basis of $V=(\Fbb^{d})^*$. Remark \ref{rmk:actiononapts} implies that given bases $(e_1,e_2,\dots, e_d)$ and $(\lambda e_1,\lambda e_2,\dots, \lambda e_d)$ for some $\lambda\in\Fbb-\{0\}$, the corresponding standard markings are equal. Therefore, given a snake $\sigma$ there exists a unique associated marking $f_\sigma$ of $\mathcal A_\sigma$ obtained by taking the ultralimit of the sequences of snake bases of $\sigma$ with respect to the sequence of triples of flags $(E_n,F_n,G_n)$. We refer to this marking as the \emph{standard marking of $\sigma$}. Thanks to Propositions \ref{prop:snakemoves}, Proposition \ref{prop:ulimend} and Lemma \ref{lem:assnakes}, we have explicit expressions for the totally nonnegative matrices $M^{\sigma'}_\sigma=\ulim M^{\sigma'}_\sigma(E_n,F_n,G_n)\in\GL(d,\Fbb)$ sending the ultralimit of the sequence of $\sigma$-bases to the ultralimit of sequences of $\sigma'$-bases. We use this fact together with Proposition \ref{prop:positiveintersection} to explicitly describe the intersections of the apartments associated to snakes.

Recall from \S \ref{ssec:snakes}, that the bottom snake $\sigma^\bott$ is the snake associated to the line decomposition $(E^{(d-i)}\oplus G^{(i-1)})^\perp$. Concretely, Lemma \ref{lem:asdiamond} and Lemma \ref{lem:astail} below say that if the snake $\sigma'$ is obtained from the snake $\sigma$ by a diamond or a tail move, the intersection between the apartments $\Acal_{\sigma^\bott}\cap\Acal_{\sigma'}$ can be obtained from the intersection $\Acal_{\sigma^\bott}\cap\Acal_\sigma$ via a restriction to a half-apartment and by a translation.

\begin{lem}\label{lem:asdiamond} Let $\sigma$ and $\sigma'$ be snakes in $\Theta^\perp_d$ such that $\sigma'$ is obtained from $\sigma$ by a diamond move at $k+1$. Suppose the intersection of $\Acal_{\sigma^\bott}\cap\Acal_\sigma$ is the image via the standard marking $f_{\sigma^\bott}$ of the set
\[
\{x\in\Abb^{d-1}\colon x_i-x_{i+1}+\alpha_i\geq -\beta_i,\ i=1,\dots, d-1\}
\]
with $\alpha_i\in\Rbb$ and $\beta_i\in \Rbb\cup\{\infty\}$. 
Then, the intersection $\Acal_{\sigma^\bott}\cap\Acal_{\sigma'}$ is the image under $f_{\sigma^\bott}$ of the set 
\[
\{x\in\Abb^{d-1}\colon x_i-x_{i+1}+\alpha_i'\geq -\beta_i',\ i=1,\dots, d-1\},
\]
where
\begin{gather*}
\alpha_i'=
\begin{cases}
\alpha_i&\text{ for } i\neq k+1\\
\alpha_i-v(X_{a,b,c})&\text{ for } i=k+1
\end{cases}
,\qquad
\beta_i'=
\begin{cases}
\beta_i&\text{ for } i\neq k,k+1\\
\min\{0,\beta_i\}&\text{ for } i= k\\
\beta_{i}+v(X_{a,b,c})&\text{ for } i=k+1
\end{cases}
\end{gather*}
 and $X_{a,b,c}$ is the triple ratio naturally associated to $\sigma$ and $\sigma'$.
\end{lem}
\begin{proof} Consider the basis change matrices $M_{\sigma^\bott}^\sigma=(m_{ij})$ from $(u^\bott_i)$ to $(u_i^\sigma)$ and $M_{\sigma^\bott}^{\sigma'}=(m'_{ij})$ from $(u^\bott_i)$ to $(u_i^{\sigma'})$ where we assume $u^\bott_1=u_1^\sigma=u_1^{\sigma'}$. Observe that $M_{\sigma^\bott}^\sigma$ and $M_{\sigma^\bott}^{\sigma'}$ are totally nonnegative as they are products of totally nonnegative matrices and they are upper triangular. Moreover, by Proposition \ref{prop:snakemoves}, we know that 
\begin{gather}\label{eq:globalsnake1}
m'_{ii}=
\begin{cases}
m_{ii} &\text{ for }i\leq k+1\\
X_{a,b,c}m_{ii} &\text{ for }i>k+1
\end{cases}
,\qquad
m'_{i,i+1}=
\begin{cases}
m_{i,i+1}& \text{ for }i<k\\
m_{ii}+m_{i,i+1} & \text{ for }i=k\\
X_{a,b,c}m_{i,i+1} &\text{ for }i>k
\end{cases}
\end{gather}
As the matrices $M_{\sigma^\bott}^{\sigma}$ and $M_{\sigma^\bott}^{\sigma'}$ are upper triangular, it follows from Proposition \ref{prop:intersection} and Corollary \ref{cor:uppertriangular} that $\Acal_{\sigma^\bott}\cap\Acal_\sigma$ and $\Acal_{\sigma^\bott}\cap\Acal_{\sigma'}$ are the images under the marking $f_{\sigma^\bott}$ of the sets
\begin{align*}
\Acal_{\sigma^\bott}\cap\Acal_\sigma\colon\left\{x\in \Abb^{d-1}\colon  x_i-x_{i+1}+v\left(\frac{m_{ii}}{m_{i+1,i+1}}\right)\geq -v\left(\frac{m_{i,i+1}}{m_{ii}}\right)\text{ for } 1\leq i<d\right\},\\
\Acal_{\sigma^\bott}\cap\Acal_{\sigma'}\colon \left\{x\in \Abb^{d-1}\colon  x_i-x_{i+1}+v\left(\frac{m'_{ii}}{m'_{i+1,i+1}}\right)\geq -v\left(\frac{m'_{i,i+1}}{m'_{ii}}\right)\text{ for } 1\leq i<d\right\}.
\end{align*}
Therefore, by Equation \ref{eq:globalsnake1}, the intersection $\Acal_{\sigma^\bott}\cap\Acal_{\sigma'}$ is the image under the marking $f_{\sigma^\bott}$ of the set of $x\in\Abb^{d-1}$ satisfying the following inequalities:
\[
\begin{cases}
x_i-x_{i+1}+v\left(\frac{m_{ii}}{m_{i+1,i+1}}\right)\geq -v\left(\frac{m_{i,i+1}}{m_{ii}}\right) & \text{ for } i< k\\ 
x_i-x_{i+1}+v\left(\frac{m_{ii}}{m_{i+1,i+1}}\right)\geq -v\left(1+\frac{m_{i,i+1}}{m_{ii}}\right)& \text{ for } i= k\\
x_{i}-x_{i+1}+v\left(\frac{m_{ii}}{m_{i+1,i+1}}\right)-v(X_{a,b,c})\geq -v\left(X_{a,b,c}\frac{m_{i,i+1}}{m_{ii}}\right)& \text{ for } i= k+1\\
x_i-x_{i+1}+v\left(\frac{X_{a,b,c}m_{ii}}{X_{a,b,c}m_{i+1,i+1}}\right)\geq -v\left(\frac{X_{a,b,c}m_{i,i+1}}{X_{a,b,c}m_{ii}}\right)& \text{ for } i>k+1\\
\end{cases}
\]
Observe that as $1+\frac{m_{k,k+1}}{m_{kk}}\in\Fbb_{\geq 0}$, its valuation is equal to $\min\left\{0,v\left(\frac{m_{k,k+1}}{m_{kk}}\right)\right\}$. The result follows by comparing the above inequalities to the inequalities defining $\Acal_{\sigma^\bott}\cap\Acal_{\sigma}$.
\end{proof}

An analogous argument as the one in the proof of Lemma \ref{lem:asdiamond} shows the following.

\begin{lem}[Asymptotic tail move]\label{lem:astail} Let $\sigma$ and $\sigma'$ be snakes in $\Theta^\perp_d$ such that $\sigma'$ is obtained from $\sigma$ by a tail move. Suppose the intersection of $\Acal_{\sigma^\bott}\cap\Acal_\sigma$ is the image under the marking $f_{\sigma^\bott}$ of the set
\[
\{x\in\Abb^{d-1}\colon x_i-x_{i+1}+\alpha_i\geq -\beta_i\text{ for }i=1,\dots,d-1\}
\]
with $\alpha_i\in\Rbb$, $\beta_i\in \Rbb\cup\{\infty\}$. Then, the intersection of $\Acal_{\sigma^\bott}\cap\Acal_{\sigma'}$ is the image under the marking $f_{\sigma^\bott}$ of the set defined by the inequalities
\[
\begin{cases}
x_i-x_{i+1}+\alpha_i\geq -\beta_i &\text{ for } i<d-1,\\
x_i-x_{i+1}+\alpha_i\geq -\min\{0,\beta_i\}&\text{ for } i=d-1.
\end{cases}
\]
\end{lem}
\begin{proof} This proof is similar to the proof of Lemma \ref{lem:asdiamond}, but it is simpler.
\end{proof}

The following theorem is the main step in the proof of Theorem \ref{intro:triple} from the introduction.

\begin{thm}\label{thm:topbot} Let $\sigma^\bott$ and $\sigma^\topp$ be the bottom and top snakes in $\Theta^\perp_d$ with snake basis $(u_i^\bott)$ and $(u_i^\topp)$, respectively. Let $\Acal_{\sigma^\bott}$ and $\Acal_{\sigma^\topp}$ denote the corresponding apartments in the $\Rbb$-Euclidean building $\Bcal_d$. Then, the intersection $\Acal_{\sigma^\bott}\cap\Acal_{\sigma^\topp}$ is the image under the marking $f_{\sigma^\bott}$ of the set of $x\in\Abb^{d-1}$ satisfying the inequalities
\begin{align}\label{ineq:main}
\begin{cases}
x_1-x_2\geq 0\\
x_2-x_3\geq \max\{0,v(X_{d-2,1,1})\}\\
x_3-x_4\geq \max\{0,v(X_{d-3,2,1}),v(X_{d-3,2,1}X_{d-3,1,2})\}\\
\ \ \ \ \ \ \ \ \ \vdots\\
x_{n-1}-x_n\geq \max\{0,v(X_{1,d-2,1}),\dots, v(X_{1,d-2,1}\cdots X_{1,1,d-2})\}
\end{cases}
\end{align}
\end{thm}
\begin{proof} We want to prove this theorem by induction on $d$. The case $d=2$ is reduced to Lemma \ref{lem:astail} as there are only two snakes in $\Theta_2^\perp$ that differ by a tail move.

\begin{figure}
\includegraphics[scale=.2]{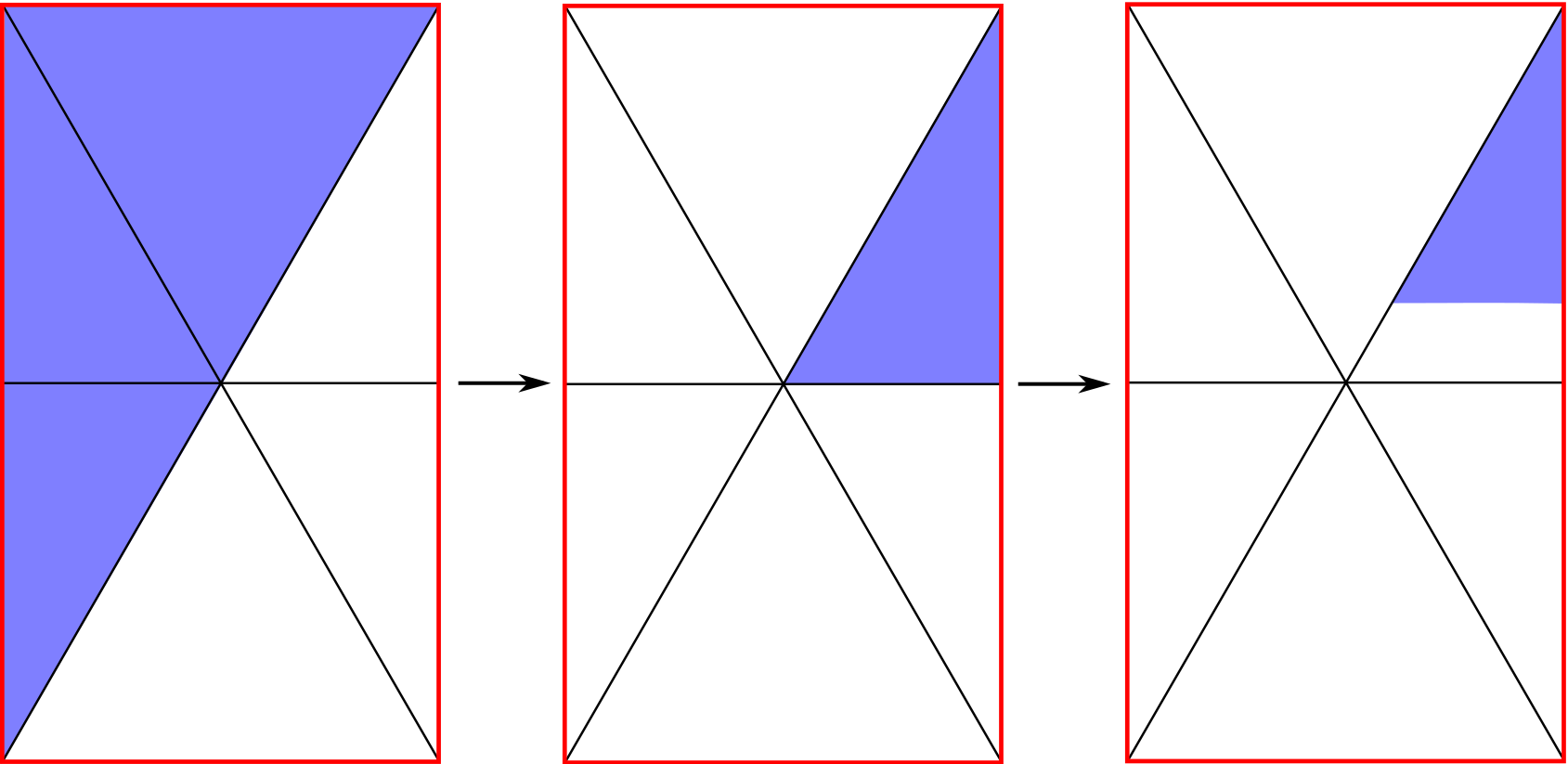}
\put (-570, 85){\makebox[0.7\textwidth][r]{\tiny{diam.}}}
\put (-455, 85){\makebox[0.7\textwidth][r]{\tiny tail}}
\caption{A pictorial version of the proof of Theorem \ref{thm:main} in dimension $d=3$.}
\label{fig:aptdynamics}
\end{figure}

\textbf{Case $d=3$:} In this case, there is only one triple ratio $X_{1,1,1}$. We can obtain the top snake from the bottom snake with a sequence of a tail move, a diamond move at $k=1$ and another tail move. Therefore, applying Lemma \ref{lem:asdiamond} and Lemma \ref{lem:astail} we have
\begin{align*}
&\begin{cases}
x_1-x_2\geq -\infty\\
x_2-x_3\geq -\infty
\end{cases}
\stackrel{\text{tail}}{\Rightarrow}
\begin{cases}
x_1-x_2\geq -\infty\\
x_2-x_3\geq 0
\end{cases}\\
&\stackrel{\text{diamond}}{\Rightarrow}
\begin{cases}
x_1-x_2\geq 0\\
x_2-x_3-v(X_{1,1,1})\geq -v(X_{1,1,1})
\end{cases}\\
&\stackrel{\text{tail}}{\Rightarrow}
\begin{cases}
x_1-x_2\geq 0\\
x_2-x_3-v(X_{1,1,1})\geq -\min\{0,v(X_{1,1,1})\}
\end{cases}
\end{align*}
and we conclude by observing that 
\[
-\min\{0,v(X_{1,1,1})\}+v(X_{1,1,1})=-\min\{0,-v(X_{1,1,1})\}=\max\{0,v(X_{1,1,1})\}.
\]
See Figure \ref{fig:aptdynamics}.

We now assume the result is true for $d-1$ and we prove it for $d$.

\textbf{Case $d>3$:} Consider the subtriangle $\Theta^\perp_{d-1}\subset\Theta^\perp_d$ as in Figure \ref{fig:mainthm}. The key observation, which follows from Lemma \ref{lem:asdiamond} and Lemma \ref{lem:astail}, is that any snake move at a vertex in $(\Theta^\perp_{d-1})^\circ$ does not affect the inequality involving the variables $x_{d-1}$ and $x_d$.

Starting with $\sigma^\bott$, by Lemmas \ref{lem:asdiamond} and \ref{lem:astail}, performing a tail move and a diamond move at $d-1$ gives us the new inequality
\[
x_{d-1}-x_d-v(X_{1,1,d-2})\geq -\min\{0,\infty\}-v(X_{1,1,d-2}).
\]
We proceed by performing diamond moves at $k$ for $k=1,2,\dots,d-2$. Again, by Lemma \ref{lem:asdiamond} these moves do not affect the last inequality, and therefore we still have
\begin{align*}
x_{d-1}-x_d-v(X_{1,1,d-2})\geq -v(X_{1,1,d-2}).
\end{align*}
We then perform a tail move in $\Theta^\perp_d$ which gives
\begin{align*}
x_{d-1}-x_d-v(X_{1,1,d-2})\geq -\min\{0,v(X_{1,1,d-2})\}
\end{align*}
and then a diamond move at $d-1$ in the triangle $\Theta^\perp_d$. This changes the last inequality to
\begin{align*}
x_{d-1}-x_d-v(X_{1,1,d-2}X_{1,2,d-3})\geq -\min\{0,v(X_{1,1,d-2})\}-v(X_{1,2,d-3})
\end{align*}
which is equivalent to
\begin{align*}
x_{d-1}-x_d-v(X_{1,1,d-2}X_{1,2,d-3}) \geq -\min\{v(X_{1,2,d-3}),v(X_{1,2,d-3}X_{1,2,d-2})\}.
\end{align*}
We conclude by iterating this procedure. More precisely, consider the level sets in $\Theta^\perp_d$ given by fixing the value of the second variable $b$ (these are horizontal lines in the discrete triangle $\Theta_d^\perp$). We can proceed by induction as $b$ varies between $1$ and $d-2$. The discussion above proves our claim for the cases $b=1$ and $2$. To simplify notation, set
\[
V_{b,a}=v(X_{1,b,d-b-1}X_{1,b-1,d-b-2}\dots X_{1,b-a,d-(b-a)-1})
\]
with $b=1,\dots,d-2$ and $a=0,\dots,b-1$. Observe that $V_{b,a}+v(X_{1,b+1,d-b-2})=V_{b,a}+V_{b+1,0}=V_{b+1,a+1}$. Assume
\begin{align*}
x_{d-1}-x_d-V_{b,b-1}\geq -\min\{V_{b,0}, V_{b,1}, \dots, V_{b,b-1}\}.
\end{align*}
Applying a tail move to the snake we obtain
\begin{align*}
x_{d-1}-x_d-V_{b,b-1}\geq -\min\{0,V_{b,0}, V_{b,1}, \dots, V_{b,b-1}\}.
\end{align*}
Hence, with a diamond move we have 
\begin{align*}
x_{d-1}-x_d-V_{b+1,b}&\geq -\min\{0,V_{b,0}, V_{b,1}, \dots, V_{b,b-1}\}-v(X_{1,b+1,d-b-2})\\
&\Updownarrow\\
x_{d-1}-x_d-V_{b+1,b}&\geq -\min\{V_{b+1,0}, V_{b+1,1}, \dots, V_{b+1,b}\}.
\end{align*}
In other words, we showed that the formula repeats itself when we apply a tail move followed by a diamond move. The result then follows because this process ends with a tail move in $\Theta_d^\perp$ which has the effect of changing the right hand side of the inequality
\begin{align*}
x_{d-1}-x_d-V_{d-2,d-3}\geq -\min\{V_{d-2,0}, V_{d-2,1}, V_{d-2,2}, \dots, V_{d-2,d-3}\}
\end{align*}
to $ -\min\{0,V_{d-2,0}, V_{d-2,1}, V_{d-2,2}, \dots, V_{d-2,d-3}\}$.
\begin{figure}
\includegraphics[scale=.3]{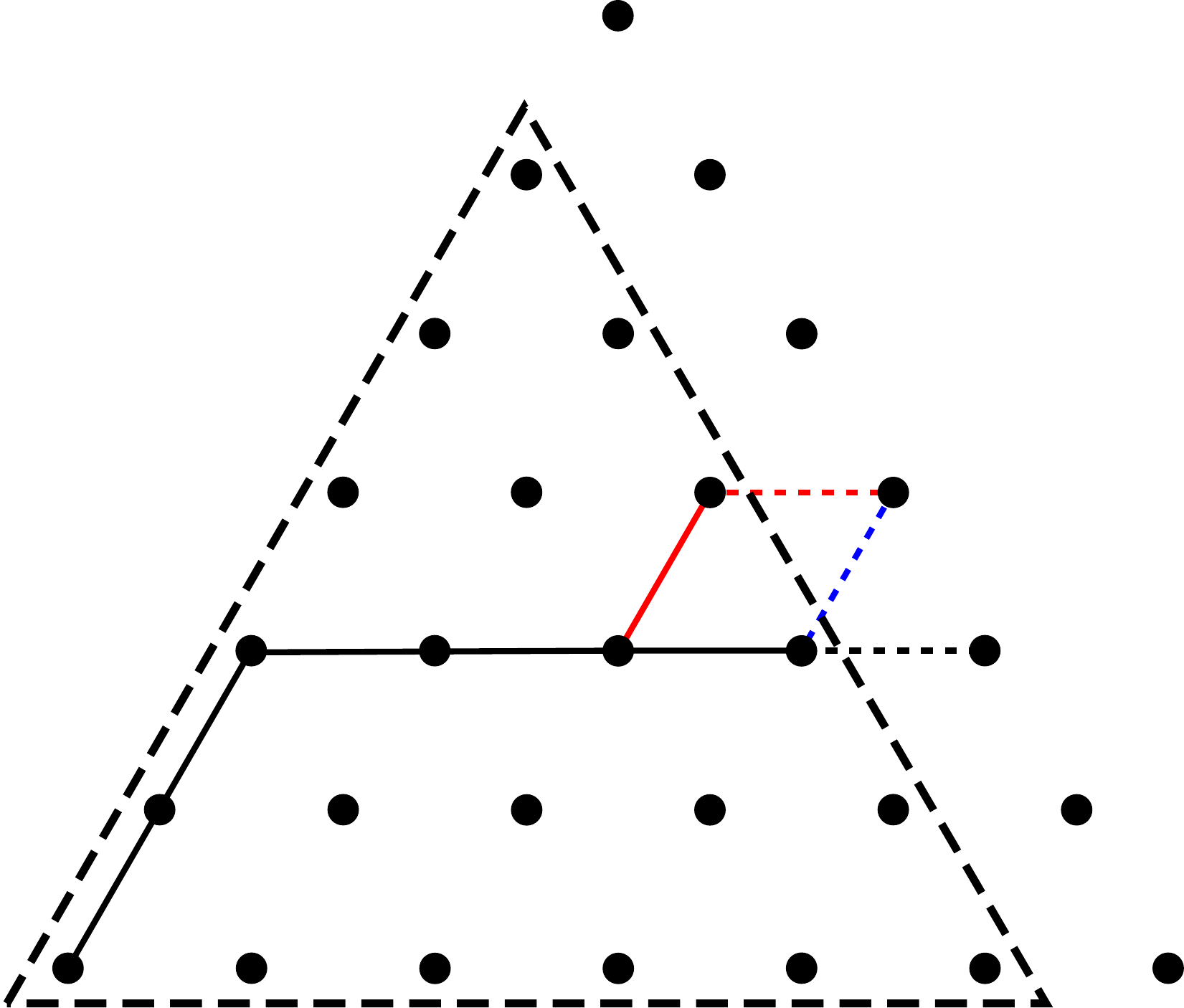}
\caption{Diamond moves at $d-1$ in the triangle $\Theta_d^\perp$ correspond to tail moves in $\Theta_{d-1}^\perp$.}
\label{fig:mainthm}
\end{figure}
\end{proof}

The edges of the discrete triangle $\Theta^\perp_d$ determine three apartments $\Acal_{EG}$, $\Acal_{GF}$ and $\Acal_{EF}$ associated to line decompositions defined by the pairs of flags $(E,G)$, $(G,F)$ and $(E,F)$, respectively. The pairwise intersections of these three apartments are non-empty by Lemma \ref{lem:snakeapt}. The following theorem expresses these intersections in terms of the valuations of the triple ratios $X_{a,b,c}$.

\begin{thm}[Theorem \ref{intro:triple}]\label{thm:main} Let $(E_n,F_n,G_n)$ be a sequence of positive triples of flags in $\Rbb^d$ such that the ultralimit $(E,F,G)$ is positive. Let $\Acal_{EG}$, $\Acal_{GF}$ and $\Acal_{FE}$ be the apartments associated to the triple $(E,F,G)$. There exists a marking $f_{EG}$ of $\Acal_{EG}$ such that
\begin{itemize}
	\item[-] the intersection $\Acal_{EG}\cap \Acal_{EF}$ is the image under $f_{EG}$ of the subset of $\Abb^{d-1}$ described by the inequalities: 
	\begin{equation}\label{eqn:mainthm1}
	\begin{cases}
	x_1-x_2\geq 0\\
	x_2-x_3\geq \max\{0,v(X_{d-2,1,1})\}\\
	x_3-x_4\geq \max\{0,v(X_{d-3,2,1}),v(X_{d-3,2,1}X_{d-3,1,2})\}\\
	\ \ \ \ \ \ \ \ \vdots\\
	x_{d-1}-x_d\geq \max\{0,v(X_{1,d-2,1}),\dots, v(X_{1,d-2,1}\cdots X_{1,1,d-2})\}
	\end{cases}
	\end{equation}
	\item[-] the intersection $\Acal_{EG}\cap \Acal_{GF}$ is the image under $f_{EG}$ of the subset of $\Abb^{d-1}$ described by the inequalities: 
	\begin{equation}\label{eqn:mainthm2}
	\begin{cases}
	x_1-x_2\leq \min\{0, v(X_{1,d-2,1}),\dots, v(X_{1,d-2,1}\cdots X_{d-2,1,1}) \}\\
	\ \ \ \ \ \ \ \ \vdots\\
	x_{d-3}-x_{d-2}\leq \min\{0,v(X_{1,2,d-3}),v(X_{1,2,d-3}X_{2,1,d-3})\}\\
	x_{d-2}-x_{d-1}\leq \min\{0,v(X_{1,1,d-2})\}\\
	x_{d-1}-x_d\leq 0
	\end{cases}
	\end{equation}
\end{itemize}
\end{thm}
\begin{proof}
The proof follows by combining Remark \ref{rmk:symtriple} and Theorem \ref{thm:topbot}.
The formula for $\Acal_{EG}\cap\Acal_{EF}$ follows at once from Theorem \ref{thm:topbot}. Permute the positive maximum span triple from $(E,F,G)$ to $(G,F,E)$. Then, we can apply Theorem \ref{thm:topbot} to the triple $(G,F,E)$ in order to find the intersection of the apartments $\Acal_{EG}\cap\Acal_{GF}$. Let $(v_i^\bott)$ denote the basis associated to the bottom snake in the discrete triangle $\Theta^\perp_d$ with respect to the maximum span triple $(G,F,E)$. By choosing $v_1^\bott=u_{d-1}^\bott$, we have that $v_i^\bott=u_{d-i}^\bott$. This allows us to explicitly relate the markings $f_{EG}$ and  $f_{GE}$. Namely, $f_{GE}$ is obtained from $f_{EG}$ by the permutation defined by
\[
w_0=\begin{pmatrix}
0 & \dots & 0 & 1\\
\vdots & \iddots & \iddots &0\\
0 & 1 & \dots &\vdots \\
1 & 0 & \dots & 0
\end{pmatrix}\in\mathfrak S_d
\]
Moreover, by Remark \ref{rmk:symtriple} we have $X_{a,b,c}(G,F,E)=X^{-1}_{c,b,a}(E,F,G)$. Therefore, by Theorem \ref{thm:topbot} and Remark \ref{rmk:actiononapts} we have that $\Acal_{EG}\cap \Acal_{FG}$ is the image via $f_{EG}$ of the set
\[
\begin{cases}
x_d-x_{d-1}\geq 0\\
x_{d-1}-x_{d-2}\geq \max\{0,-v(X_{1,1,d-2})\}\\
x_{d-2}-x_{d-3}\geq \max\{0,-v(X_{1,2,d-3}),-v(X_{d-3,2,1}X_{2,1,d-3})\}\\
\ \ \ \ \ \ \ \ \ \ \ \ \vdots\\
x_2-x_1\geq \max\{0,-v(X_{1,d-2,1}),\dots, -v(X_{1,d-2,1}\cdots X_{d-2,1,1})\}
\end{cases}
\]
It is easy to see that these inequalities are equivalent to the ones in the statement of the theorem.
\end{proof}


\begin{remark} The marking $f_{EG}$ from Theorem \ref{thm:main} determines a preferred point $f_{EG}(1/d,1/d,\dots, 1/d)\in\Acal_{EG}$. Geometrically, Theorem \ref{thm:main} says that the intersections $\Acal_{EG}\cap \Acal_{EF}$ and $\Acal_{EG}\cap\Acal_{FG}$ are Weyl sectors $f_{EG}(\mathfrak C_1)$ and $f_{EG}(\mathfrak C_2)$ contained in the opposite Weyl sectors based at $f_{EG}(1/d,1/d,\dots, 1/d)$. Note that the triple intersection $\Acal_{EG}\cap\Acal_{FE}\cap\Acal_{GF}$ is a point when all triple ratios have valuation equal to zero. For $d=2$ this recovers the fact that if three apartments (lines) in an $\Rbb$-tree intersect pairwise along half-lines, then they form a tripod. For $d=3$, Theorem \ref{thm:main} was proved by Parreau \cite{Par4} in greater generality, but with different methods.
\end{remark}

\begin{remark} Theorem \ref{thm:main} and Remark \ref{rmk:symtriple} suffice to describe the intersections $\Acal_{E',G'}\cap\Acal_{F',G'}$ for any $E',F',G'$ with $\{E',F',G'\}=\{E,F,G\}$. 
\end{remark}


\subsection{Shearing in $\Bcal_{d}$}

The analogous of Theorem \ref{thm:main} for the positive ultralimit of a sequence of positive quadruple of flags $(E_n,F_n,G_n,H_n)$ follows from Proposition \ref{prop:shearing}. Let us ease notation by setting
\[
Z_i:=\ulim Z_i(E_n,F_n,G_n,H_n)
\]
which we are assuming to be a positive element in $\Fbb$.

\begin{thm}\label{thm:shearing} Let $(E_n,F_n,G_n,H_n)$ be a sequence of positive quadruples of flags in $\Rbb^d$ with positive ultralimit the quadruple of flags $(E,F,G,H)$ in $\Fbb^d$. Consider the markings $f_{EG}$ and $f'_{EG}$ of the apartment $\Acal_{EG}$ obtained by applying Theorem \ref{thm:main} to the sequence of positive triples of flags $(E_n,F_n,G_n)$ and $(E_n,H_n,G_n)$, respectively. Then, the element
\[
w_{(E,F,G,H)}:=f_{EG}^{-1}\circ f'_{EG}\in W_{\text{aff}}
\]
is the translation by the unique vector $(z_1,z_2,\dots,z_d)\in\mathbb V^{d-1}$ such that $z_i-z_{i+1}=-v(Z_{d-i})$.
\end{thm}
\begin{proof} Recall that the sequence of flags $(E_n,F_n,G_n)$ determines a sequence of bases $(u_{i,n})$ for the line decomposition associated to the flags $(E_n,G_n)$. Likewise, the sequence of flags $(E_n,H_n,G_n)$ determines a sequence of bases $(U_{i,n})$ for the line decomposition associated to the flags $(E_n,G_n)$. Thanks to Lemma \ref{lem:assnakes}, the ultralimits of the basis $(u_{i,n})$ and $(U_{i,n})$ define bases $(u_i)$ and $(U_i)$ of the vector space $V=(\Fbb^d)^*$. The markings $f_{EG}$ and $f'_{EG}$ in Theorem \ref{thm:main} are the standard markings of the bases $(u_i)$ and $(U_i)$, respectively. Proposition \ref{prop:shearing} implies that $U_i=Z_1Z_2\cdots Z_{d-i}u_i$. Applying Remark \ref{rmk:actiononapts}, we have that for all $(x_1,x_2,\dots,x_d)$, 
\[
f_{(u_i)}^{-1}\circ f_{(U_i)}(x_1,x_2,\dots,x_d)=(x_1+z_1,x_2+z_2,\dots, x_d+z_d)
\]
where $(z_1,z_2,\dots,z_d)\in\mathbb V^{d-1}$ is such that
\begin{align*}
z_i-z_{i+1}&=v((Z_1Z_2\cdots Z_{d-i})^{-1})-v((Z_1Z_2\cdots Z_{d-i-1})^{-1})\\
&=-v(Z_{d-i})
\end{align*}
which is what needed to be proved.
\end{proof}

\begin{remark} Theorem \ref{thm:shearing} is equivalent to \cite[Prop. 4.5]{Par4}. The reader should be aware of the small difference between the double ratios and the edge parameters as explained in \cite[\S 2.6]{Par4}.
\end{remark}

\bibliography{mybib}
\bibliographystyle{plain}

\end{document}